\documentclass[11pt]{article}
\usepackage{geometry}                
\usepackage{graphicx}
\usepackage{amssymb,amsmath,amsthm,bbm,xcolor}
\usepackage{mhequ}
\usepackage{mathabx}
\usepackage{natbib}
\usepackage[colorlinks,citecolor=blue,urlcolor=blue]{hyperref}

\theoremstyle{plain}
\newtheorem{lemma}{Lemma}[section]
\newtheorem{theorem}{Theorem}[section]

\usepackage{tikz}
\usepackage{rotating}
\usepackage{amsmath}

\usepackage{times}
\usepackage{bm}
\usepackage{natbib}
\usepackage{xcolor}
\usepackage{verbatim}
\usepackage{graphicx,epsfig,caption}
\usepackage{amsmath,amssymb,latexsym, amsfonts, amscd}
\usepackage[plain,noend]{algorithm2e}

\usepackage{threeparttable}
\usepackage{multirow}
\usepackage{amsfonts}
\usepackage{amsmath,amssymb,setspace}

\usepackage{comment}
\newcommand{\+}[1]{\ensuremath{\boldsymbol {#1}}}

\usepackage{subfigure}

\newcommand{\be}{\begin{equs}}
	\newcommand{\ee}{\end{equs}}
\newcommand{\bpm}{\begin{pmatrix}}
	\newcommand{\epm}{\end{pmatrix}}

\newtheorem{proposition}{Proposition}

\usepackage{enumitem}

\newtheorem{assumption}{Assumption}
\newtheorem{definition}{Definition}
\graphicspath{{./pictures/}}

\usepackage{blindtext}

        \title{High-Dimensional Linear Regression via Implicit Regularization}

\author{Peng Zhao, Department of Statistics, Texas A\&M University
	\and Yun Yang, Department of Statistics, University of Illinois Urbana-Champaign
	\and  Qiao-Chu He, Southern University of Science and Technology}

\date{}
\begin{document}

        \maketitle

        \begin{abstract}
                Many statistical estimators for high-dimensional linear regression are $M$-estimators, formed through minimizing a data-dependent square loss function plus a regularizer. This work considers a new class of estimators implicitly defined through a discretized gradient dynamic system under over-parametrization. We show that under suitable Restricted Isometry conditions, over-parametrization leads to implicit regularization: if we directly apply gradient descent to the residual sum of squares with sufficiently small initial values, then under some proper early stopping rule, the iterates converge to a nearly sparse rate-optimal solution that improves over explicitly regularized approaches.
                In particular, the resulting estimator does not suffer from extra bias due to explicit penalties, and can achieve the parametric root-$n$ rate when the signal-to-noise ratio is sufficiently high.
                We also perform simulations to compare our methods with high dimensional linear regression with explicit regularization. Our results illustrate the advantages of using implicit regularization via gradient descent after over-parametrization in sparse vector estimation.
        \end{abstract}

        \section{Introduction}

        \label{sec:intro}
        In high dimensional linear regression
        $y=X\beta^\ast+w$
        with noise $w=(w_1,\ldots,w_n)^T\in\mathbb R^n$, where each component of noise follows an independent sub-Gaussian distribution with scale bounded by $\sigma$, the goal is to parsimoniously predict the response $y\in \mathbb R^n$ as a linear combination of numerous  covariates $X=(X_1,X_2,\ldots,X_p)\in\mathbb R^{n\times p}$, and conduct statistical inference on the coefficients $\beta^\ast=(\beta_1^\ast,\ldots,\beta_p^\ast)^T\in\mathbb R^p$ \citep{tibshirani1996regression,donoho2006compressed}. By leveraging certain lower dimensional structure in the regression coefficient $\beta^\ast\in\mathbb R^p$ such as a sparsity constraint $s=\|\beta^\ast\|_0\ll n$, where $\|\beta^\ast\|_0$ counts the number of nonzero elements in $\beta^\ast$, the number of covariates $p$ may be substantially larger than the sample size $n$. Because of the intrinsic computational hardness in dealing with the $\ell_0$ metric, people instead use different metrics as surrogates, and cast the estimation problem into various convex or non-convex optimization problems. Many approaches have been proposed for high dimensional regression by solving certain penalized optimization problem, including basis pursuit \citep{chen2001atomic}, the lasso \citep{tibshirani1996regression}, the Dantzig selector \citep{candes2007dantzig}, smoothly clipped absolute deviation \citep{fan2001variable}, minimax concave penalty \citep{zhang2010nearly} and so on. In this work, we focus on the estimation of $\beta^\ast\in\mathbb R^p$ without explicitly specifying a penalty. Recent work~\citep{hoff2017lasso} showed that through a change-of-variable via Hadamard product over-parametrization, the non-smooth convex optimization problem for the lasso:
        \begin{equation}\label{Eqn:CS}
        \min_{\beta} \frac{1}{2n}\|X\beta-y\|^2+\lambda \|\beta\|_{1}, \quad\mbox{with }\|\beta\|_1:\,=\sum_{j=1}^p|\beta_j|
        \end{equation}
        can be reformulated as a smoothed optimization problem at a cost of introducing non-convexity. Due to the smoothness feature, simple and low-cost first-order optimization methods such as gradient descent and coordinate descent can recover $\beta^\ast$. Despite the non-convexity and exponentially many stationary points induced by the change-of-variable, these first-order algorithms exhibit encouraging empirical performance~\citep{hoff2017lasso}. 
        
        In this work, we consider the same Hadamard product over-parametrization $\beta = g\circ l$ as in \cite{hoff2017lasso}, where $g, \,l\in\mathbb R^p$ and $\circ$ denotes the Hadamard product operator, which
        performs element-wise product between two vectors. Instead of solving the penalized optimization problem~\eqref{Eqn:CS}, we consider directly applying the gradient descent to the quadratic loss function
        \begin{equation}\label{eq_opt}
        f(g,l)=\frac{1}{2n}\,\|X(g\circ l)-y\|^2.
        \end{equation}
        In the special case of noiseless responses when $\sigma=0$, minimizing $f(g,\,l)$ jointly over $(g,\,l)$ is a highly non-convex optimization problem with exponentially many saddle points. Interestingly, our theory indicates that by initializing $g$ and $l$ arbitrarily close to zero, a properly tuned gradient method converges to least $\ell_1$-norm solution when the design matrix $X$ satisfies the prominent Restricted Isometry Property~\citep{candes2008restricted}, thereby inducing algorithmically implicit regularization. In the general case when $\sigma>0$, we show that by combining gradient method with early stopping~\citep{zhang2005boosting,raskutti2014early}, the resulting estimator overcomes the large bias~\citep{vito2005learning,yao2007early} suffered by usual explicitly penalized approaches and leads to more accurate estimation.
        In particular, we show that by iteratively updating $(g,l)$ for a certain number steps, the resulting estimator can adapt to an optimal convergence rate of $\sqrt{s/n}$ when all signals are relatively strong. When both strong signals and weak signals exist, our estimator attains the rate $\sqrt{s_1/n}+\sqrt{s_2 \log p/n}$, with $s_1, s_2$ denoting the number of strong signals and weak signals respectively. Our numerical studies also illustrate encouraging improvements over existing explicitly regularization-based approaches.

        Our work complements the recent surge of literature on implicit regularization in first-order iterative methods for solving non-convex optimization in machine learning~\citep{gunasekar2017implicit,gunasekar2018characterizing,soudry2018implicit,li2018algorithmic} that exclusively focused on noiseless or perfectly separable data.
        We show that in the context of sparse vector estimation, through a simple change-of-variable, a simple gradient descent initialized near-zero induces sparsity, and solves the non-convex optimization problem with provable guarantees. More broadly, this paper presents a new way of designing statistical estimators---unlike usual estimators such as $M$-estimators and $Z$-estimators that are defined either as optimizers or roots of some criterion functions or system of equations, the estimator considered in this paper is implicitly defined through a discretized gradient dynamic system. Such an algorithmically induced estimator has several advantages over existing estimation procedures for high dimensional linear regression under sparsity constraints. First, the proposed implicit regularized estimator suffers less bias than some penalty-induced estimators, leading to more accurate estimation.
        Also, despite the non-convexity nature, our method has a natural initialization that provably leads to the optimal solution.
        In comparison, state-of-the-art $M$-estimators based on non-convex penalties require stringent conditions on their initializations to avoid bad local minima with bad estimation errors.

        \section{Background and Our Method}

        \subsection{Setup and notations}
        Recall that $\beta^\ast$ is the unknown $s$-sparse signal in $\mathbb{R}^{p}$ to be recovered. Let $S\subset\{1,\ldots,p\}$ denote the index set that corresponds to the nonzero components of $\beta^\ast$, and the size $|S|$ of $S$ is then $s$. 
        For two vectors $g, l\in \mathbb{R}^{p}$, we call $\beta = g\circ l \in\mathbb R^p$ as their Hadamard product, whose components are $\beta_j = g_jl_j$ for $j=1,\ldots p$. For two vectors $a,b\in\mathbb R^p$, we use the notation $a\geq b$ to indicate element-wise ``great than or equal to''.
        When there is no ambiguity, we use $\beta^2=\beta\circ \beta$ to denote the self-Hadamard product of $\beta$.  For a function $f:\mathbb R^p\times \mathbb R^p \to \mathbb R$, $(g,\,l)\mapsto f(g,\,l)$, we use $\nabla_g f$ and $\nabla_l f$ to denote its partial derivative relative to $g$ and $l$, respectively. 
        For any index set $J\subset \{1,\ldots,p\}$ and vector $a\in \mathbb R^p$, we use $a_J=(a_j:\, j\in J)$ to denote the sub-vector of $a$ formed by concatenating the components indexed by $J$. Let $\mathbf 1\in\mathbb R^p$ denote the vector with all entries as $1$, and $I$ as the identity matrix in $\mathbb R^p$.  Let $I_J$ be the diagonal matrix with one on the $j$th diagonal for $j\in J$ and $0$ elsewhere. For a vector $a\in\mathbb R^p$, we use $\|a\|$ to denote its vector-$\ell_2$-norm, and $\|a\|_\infty=\max_{j}|a_j|$ its $\ell_\infty$-norm. Let $\mbox{Unif}(a,b)$ denote the uniform distribution over interval $(a,b)$. For a symmetric matrix $A$, let 
        $\lambda_{\min}(A)$ denote its smallest eigenvalue. For two sequences $\{a_n\}$ and $\{b_n\}$, we use the notation $a_n\lesssim b_n$ or $a_n\gtrsim b_n$ to mean there exist some constants $c$ and $C$ independent of $n$ such that $a_n \leq Cb_n$ or $a_n \geq cb_n$ for all $n>0$, respectively, and $a_n\asymp b_n$ to mean $a_n\lesssim b_n$ and $b_n \lesssim a_n$.
        
        \subsection{Related literature}
        \cite{li2018algorithmic} studies the theory for implicit regularization in matrix sensing, which requires the data to be perfectly measured and has different geometric structures as linear regression. \cite{hoff2017lasso} considers the Hadamard product parametrization to optimize the parameters in high-dimensional linear regression. In particular, their objective function involves an $\ell_2$ penalty on $(g,l)$ that is equivalent to the $\ell_1$ penalty on $\beta$, and the solution is precisely the lasso solution. During the revision of our paper, we learned that a recent work~\citep{vaskevicius2019implicit} also studied high-dimensional linear regression via implicit regularization via a slightly different parameterization. Our work is different from~\cite{vaskevicius2019implicit} in many aspects. A detailed comparison between the two works is provided in Section F of the supplementary material.
        
        \subsection{Gradient descent with Hadamard product parametrization}
        As we mentioned in the introduction, we consider augmenting the $p$-dimensional vector $\beta$ into two $p$-dimensional vectors $g,\,l$ through $\beta=g\circ l$. Instead of solving the lasso problem with $\beta$ replaced with $g\circ l$, we consider directly applying gradient descent to the quadratic loss function $f(g,l)=(2n)^{-1}\|X(g\circ l)-y\|^2$. 
        In particular, we apply the updating formula $g_{t+1}=g_t-\eta \nabla f_g(g_t,\,l_t)$, $l_{t+1}=l_t-\eta \nabla_l f(g_t,l_t)$, with random initial values $[g_0]_j\overset{iid}{\sim}\mbox{Unif}(-\alpha,\alpha)$,  $[l_0]_j\overset{iid}{\sim}\mbox{Unif}(-\alpha,\alpha)$ for $j=1,\ldots,p$ (notice that $(0,0)$ is a saddle point of the objective function, so we need to apply a small perturbation $\alpha$ on the initial values).
        This leads to the standard gradient descent algorithm, and the iterates $(g_{t+1},l_{t+1})$ tend to converge to a stationary point $(g_\infty, l_{\infty})$ of $f(g,l)$ that satisfies the first order optimality condition $\nabla f_g(g_\infty,\,l_\infty) = 0$ and $\nabla f_l(g_\infty,\,l_\infty) = 0$. However, stationary points of $f(g,l)$ can be local minimum, local maximum, or saddle points (when the Hessian matrix $\nabla^2_{g,l} f(g,l)$ contains both positive and negative eigenvalues).
        The following result provides the optimization landscape of $f(g,l)$, showing that $f(g,l)$ does not have local maximums, all its local minimums are global minimum, and all saddle points are strict. The strict saddle points are saddle points with the negative smallest eigenvalues for the Hessian matrix.

        \begin{lemma}  \label{Lem:global_min}
                $f(g,l)=(2n)^{-1}\|X(g\circ l)-y\|^2$ does not have local maximum, and all its local minimums are global minimum. In particular, $(\bar g, \bar l\,)$ is a global minimum of $f(g,l)$ if and only if
                \begin{equation*}
                X^T\big\{X(\bar g\circ \bar l)-y\big\} = 0.
                \end{equation*}
                In addition, any saddle point $(g^\dagger,l^\dagger)$ of $f(g,l)$ is a strict saddle, that is, 
                $\lambda_{\min}\big(\nabla^2_{g,l} f(g^\dagger,l^\dagger)\big)<0$.
        \end{lemma}
        
        \noindent According to the first order condition associated with $f(g,l)$
        \begin{equation*}
        g\circ \big[X^T\big\{X(g\circ l)-y\big\}\big]=l\circ \big[X^T\big\{X(g\circ l)-y\big\}\big] = 0,
        \end{equation*}
        there could be exponentially many saddle points as a solution to this equation, for example, for those $(g,l)$ satisfying
        \begin{equation*}
        g_A=l_A =0\in\mathbb R^{|A|}, \qquad\mbox{and}\qquad \big[X^T\big\{X(g\circ l)-y\big\}\big]_{A^c} = 0\in \mathbb R^{p-|A|},
        \end{equation*}
        for any non-empty subset $A$ of $\{1,\ldots,p\}$. Consequently, the gradient descent algorithm may converge to any of these bad saddle points. To see this, if we initialize $(g,l)$ in a way such that the components in the index set $A$ are zero, or $[g_0]_A=[l_0]_A=0$, then these components will remain zero forever in the gradient iterations. Fortunately, the following result implies that as long as we use a random initialization for $(g,l)$ with continuous probability density functions over $\mathbb R^{2p}$ as in the gradient descent algorithm, then the gradient descent almost surely converges to a global minimum.
        
        \begin{lemma}\label{Lem:GD_global_min}
                Suppose the step size $\eta$ is sufficiently small. Then with probability one, the above gradient descent algorithm converges to a global minimum of $f(g,l)$.
        \end{lemma}
        
        In the low-dimensional regime where the design matrix $X$ has full column rank, the solution $\bar \beta$ to the normal equation $X^T(X\beta-y) = 0$ is unique, which is also the least-squares estimator. Under this scenario, Lemma~\ref{Lem:global_min} and Lemma~\ref{Lem:GD_global_min} together certify that the above gradient descent algorithm will converge to this optimal least-squares estimator.  However, in the high-dimensional regime, which is the main focus of the paper, the normal equation $X^T(X\beta-y) = 0$ has infinitely many solutions, and it is not clear which solution the algorithm tends to converge to. For example, if we consider instead applying the gradient descent to the original parameter $\beta$ in the objective function $(2n)^{-1}\|X\beta-y\|^2$ with initialization $\beta_0=0$, then the iterates will converge to the minimal $\ell_2$-norm solution of the normal equation. Interestingly, as we will illustrate in the following, under the Hadamard parametrization, the gradient descent algorithm now tends to converge to the minimal $\ell_1$-norm solution under certain conditions on initialization and the design matrix, thereby inducing sparsity implicitly.
        
        \subsection{Gradient descent converges to sparse solution}\label{Sec:Heuristic}
        In this subsection, we provide two different perspectives for understanding the following informal statement on the behavior of simple gradient descent for the loss function $f(g,\,l)$ defined in~\eqref{eq_opt} under the Hadamard product parameterization $\beta =g\circ l$.      We assume that the noise variance $\sigma^2=0$ throughout this subsection and will discuss the case $\sigma^2>0$ later.

        \emph{Informal Statement: If we initialize the algorithm to be arbitrarily close to $g=l=0$, then under suitable conditions on the design $X$, a simple gradient descent converges to a solution of basis pursuit problem:
                \begin{equation}\label{bs}
                \min_{\beta\in\mathbb R^p}{\|\beta\|_1}  \quad \mbox{subject to}  \quad X\beta =y.
                \end{equation}}

        Consider the under-determined system $X\beta=y$, where $X\in\mathbb R^{n\times p}$ has full row rank.
        Our first intuition comes from the fact that a zero-initialized gradient descent algorithm over $\beta\in\mathbb R^p$ for solving $
        \min_{\beta\in\mathbb R^p} \,\|X\beta - y\|^2/(2n):\,= h(\beta) $
        finds a minimal $\ell_2$-norm solution to $X\beta=y$. Then under the Hadamard product parameterization $\beta =g\circ l$, the fact that gradient descent tends to find the  minimal $\ell_2$-norm solution suggests (this is not rigorous) that the gradient descent algorithm for jointly minimizing $f(g,\,l)$ over $(g,\,l)$ tends to converge to a solution to $X(g\circ l) =y$ with a minimal $\ell_2$-norm $\sqrt{\|g\|^2+\|l\|^2}$. However, a minimal $\ell_2$-norm solution to $X(g\circ l) =y$ must satisfy $|g_j|=|l_j|$ for each $j=1,\ldots,p$ (otherwise we can always construct another solution with strictly smaller $\ell_2$-norm), which implies $\sqrt{\|g\|^2+\|l\|^2} = \sqrt{2 \|g\circ l\|_1 }=\sqrt{2\|\beta\|_1}$. As a consequence, $\beta_{\infty}=g_{\infty}\circ l_{\infty}$ should be the minimal $\ell_1$-norm solution to $X\beta = y$.

        Another way to understand the difference in the evolutions of gradient descents for $f(g,l)$ and $h(\beta)$ is by noticing that the gradient $\nabla_{g_j} f(g,l) = l_j\cdot \nabla_{\beta_j} h(\beta)\big|_{\beta =g\circ l}$ in the new parametrization, for each $j=1,\ldots,p$, has an extra multiplicative factor of $l_j$ than the gradient $\nabla_{\beta_j} h(\beta)$ in the usual least squares of minimizing $h(\beta)$. It is precisely this extra multiplicative factor $l_j$ that helps select important signals, which are the nonzero regression coefficients and prevent unimportant signals, which are the zero regression coefficients, to grow too fast at the early stage of the evolution when both $g$ and $l$ are close to zero.
        Our second perspective comes from considering the limiting gradient dynamical system of the problem with an infinitesimally small step size $\eta$, which is studied by \cite{gunasekar2017implicit} and \cite{gunasekar2018characterizing}. In particular, the behavior of this limiting dynamical system is captured by the ordinary differential equations
        \begin{equation}\label{de}
        \begin{cases} 
        \ \dot{g}(t)=-\big\{ X^{T}r(t)\big\} \circ l(t),\\ 
        \ \, \dot{l}(t)=-\big\{X^{T}r(t)\big\}\circ g(t),
        \end{cases}\quad\mbox{with initialization}\quad 
        \begin{cases} 
        \ g(0)=\alpha\mathbf 1,\\ 
        \ \, l(0)= 0,
        \end{cases}
        \end{equation}
        where $r(t)=n^{-1}\big[X\{g(t)\circ l(t)\}-y\big]\in\mathbb R^p$, and for simplicity we  fixed the initialization. As shown in the Section A.1 of the supplementary material, the limiting point condition $X\beta_{\infty}=y$ coincides with the Karush-Kuhn-Tucker condition for the basis pursuit problem~\eqref{bs}. Moreover, under the same solution continuity assumption, if there is additive noise on $y$, then  the limiting point condition matches the Karush-Kuhn-Tucker condition for the lasso~\eqref{Eqn:CS}.
        Again, the above two arguments are based on the hard-to-check solution continuity assumption. In the next section, we provide a formal proof of the result without making this assumption, albeit under a somewhat strong Restricted Isometry Property on $X$.  
        
        Finally, as suggested by one reviewer, our current framework can be extended to allow $\ell_1$ plus $\ell_2$ type elastic net regularization by adding a $\lambda \|g \circ l \|^2/2$ penalty term to the current sum of squares objective function. Details about its formulation can be found in Subsection~A.4 in the supplementary material.

        \subsection{Gradient descent with early stopping}
        In this subsection, we consider the general case where the response $y$ contains noise: $\sigma^2\neq 0$. In particular, we combine early stopping, a widely used implicit regularization technique \citep{zhang2005boosting,raskutti2014early}, with the gradient descent to prevent overfitting. Algorithm~\ref{alg2} below provides the detailed implementation. Here in the algorithm, we adopted the hold-out method where the algorithm terminates once the testing error on the validation data started to increase, as the testing error is expected to display a U-shaped curve similar to the estimation error shown in Fig.~\ref{fig:31}.

        \begin{algorithm}
                Data: Training design matrix $X\in\mathbb R^{n\times p}$,\, measurement vector $y\in\mathbb R^n$, validation data $X'$, $y'$, \\
                \qquad initialization magnitude $\alpha$, step size $\eta$, and maximal number of iterations $T_{max}$; \\
                \qquad  Initialize variables $[g_0]_j\overset{iid}{\sim}\mbox{Unif}(-\alpha,\alpha)$,  $[l_0]_j\overset{iid}{\sim}\mbox{Unif}(-\alpha,\alpha)$ for $j=1,\ldots,p$, and iteration \\ \qquad number $t=0$;\\
                While {$t<T_{max}$\\}{
                        \qquad  ${g}_{t+1}=g_t-\eta \ l_t \circ \big[n^{-1}\,X^{T}\big\{X(g_t\circ l_t)-y\big\}\big]$;\\ 
                        \qquad  $\, {l}_{t+1}=l_t-\eta \ g_t \circ \big[n^{-1}\,X^{T}\big\{X(g_t\circ l_t)-y\big\}\big]$; \\
                        \qquad  $\, t=t+1$;\\
                }
                Output: Choose the first ${\tilde t}$ such that $\|X'(g_{{\tilde t}}\circ l_{{\tilde t}})-y'\|< \|X'(g_{{\tilde t}+1}\circ l_{{\tilde t}+1})-y'\|$ or \\ \qquad  $\|X'(g_{{\tilde t}}\circ l_{{\tilde t}})-y'\|$ is minimized over all iterations. Then output the final estimate $\widehat \beta=g_{{\tilde t}}\circ l_{{\tilde t}}$.
                \caption{Gradient Descent for Linear Regression with Validation Data \label{alg2}} 
        \end{algorithm}

        \begin{figure}
                \subfigure[Error vs. Iteration]{\includegraphics[width=0.3\linewidth]{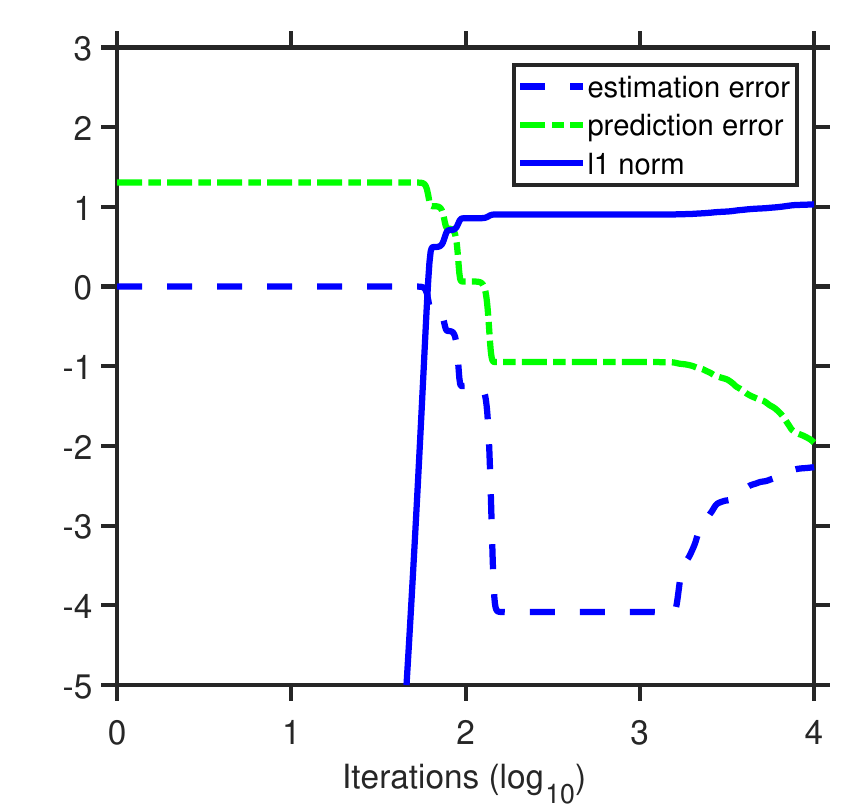}}
                \subfigure[Error vs. $\ell_1$ norm] {
                        \includegraphics[width=0.3\linewidth]{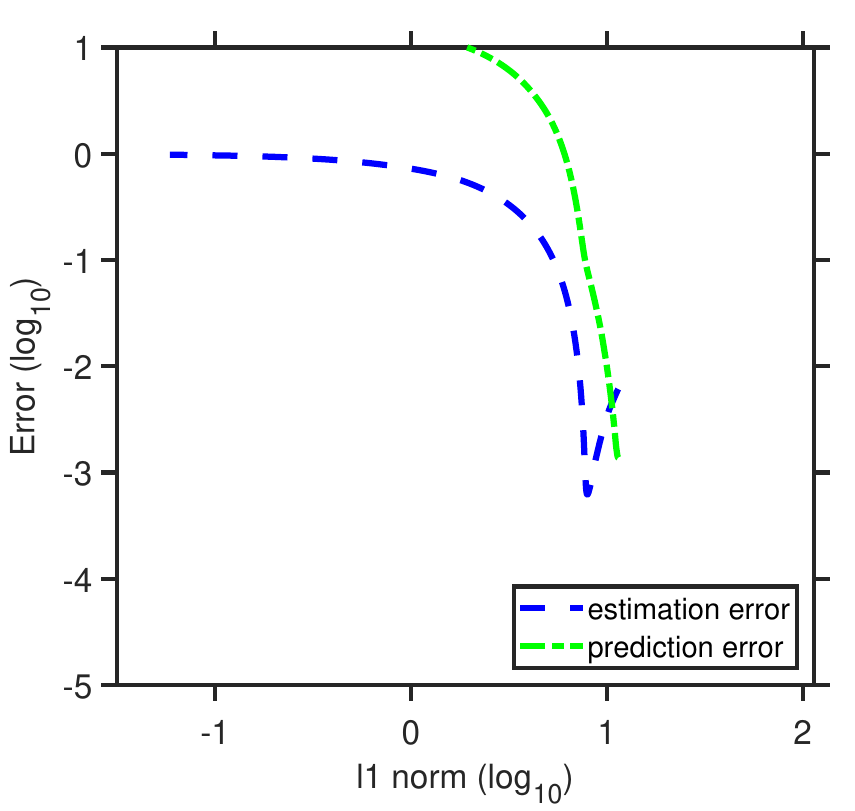}}
                \subfigure[Error, $\ell_1$ norm vs. $\lambda$ ]{         \includegraphics[width=0.3\linewidth]{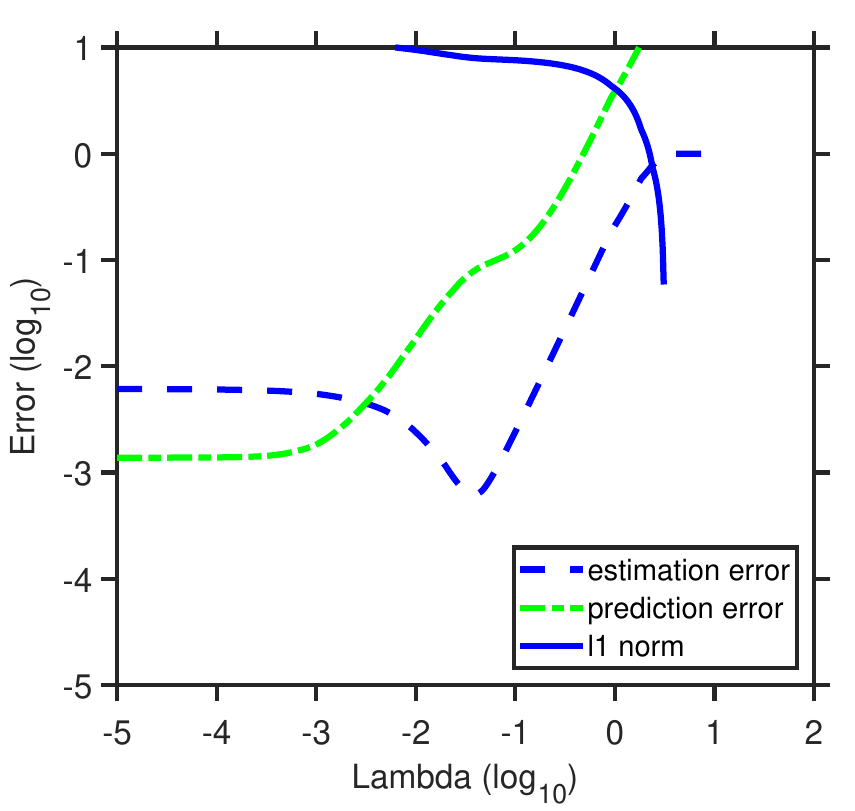}
                }
                \caption{Panel (a) is a log-log plot of standardized estimation error $\|\widehat \beta-\beta^\ast\|^2/\|\beta^\ast\|^2$, mean prediction error $\sqrt{\|\widehat y-y\|^2/n}$ and $\ell_1$ norm of estimated coefficients  versus iteration number $t  $ for gradient descent. Panel (b) is a log-log plot of  estimation error and prediction error of the optimal solution of lasso versus different $\ell_1$ norm levels; Panel (c) is  a log-log plot of estimation, prediction error and $\ell_1$ norm of the optimal solution of lasso versus different regularization parameter $\lambda$.}\label{fig:31}
        \end{figure}
        
        Due to the connection of our method with the basis pursuit problem~\eqref{bs}, one may naturally think that our method in the noisy case should be equivalent to a basis pursuit denoising problem:
        \begin{equation}\label{bsdn}
        \min \|\beta\|_1  \quad \mbox{subject to} \quad \|X \beta -y\| \leq \varepsilon,
        \end{equation}
        with some error tolerance level $\varepsilon$ depending on the stopping criterion, which is equivalent to the lasso. Surprisingly, a simulation example below shows that the iterate path of the gradient descent Algorithm~\ref{alg2} contains estimates with much smaller errors than the lasso. Precisely, we adopt the simulation setting S2 in Section~\ref{sec:simu} . As comparisons, we also report the lasso solution path as a function of the regularization parameter $\lambda$ and $\ell_1$ norm level. For our gradient descent algorithm, we set $\alpha = 10^{-5}$ in the random initialization. From Figure~\ref{fig:31}, when the iteration number is around $1000$, even though the prediction errors in panel~(a) of our algorithm and the optimal lasso solution in panel~(b) are almost the same, the estimation error in panel~(a) of our method is significantly lower than that of the lasso, illustrating the occurrence of the large bias of the lasso. Moreover, the stabilized region of our method in panel~(a) corresponding to the optimal stopping time is relatively wide. Therefore, the performance of implicit regularization with early stopping tends to be robust to the stopping criterion. In addition,  from panel (a) in Fig.~\ref{fig:31}, we can see that as the number of iterations increases, the $\ell_1$ norm of the iterate also increases. When the logarithm of the iteration number is within $(2.2,3.2)$, the $\ell_1$ norm of the estimated coefficients tends to be stabilized around the $\ell_1$ norm of the true $\beta^\ast$, corresponding to the most accurate estimation region in panel~(a) of Fig.~\ref{fig:31}. In contrast, as we can see from panels ~(b) and (c) of Fig.~\ref{fig:31}, the estimation error is very sensitive in the regularization parameter/$\ell_1$ norm domain --- the region corresponds to the smallest estimation accuracy is very narrow. A small change in the $\ell_1$ norm or regularization parameters in the solution leads to a drastic deterioration in the estimation accuracy.   The above simulation results illustrate that our implicit regularization method is unaffected, or at least less suffered from the large bias.

        We may also employ implicit regularization as a refining technique to improve upon an estimator $\beta^\dagger$ based on explicit regularization, such as the lasso, by transforming the data so that our algorithm is ``initialized" near $\beta^\dagger$. Algorithm~\ref{alg:1} below summarizes the procedure. Numerical comparisons with Algorithm~\ref{alg2} is reported in subsection E of the supplementary material. In situations where signal strengths in $\beta^\ast$ are highly-unbalanced, meaning that the ratio between the largest and the smallest $|\beta_j|$'s is large, Algorithm~\ref{alg2} may need $O(n)$ iterations to converge according to Theorem~\ref{thm2}. In comparison, Algorithm~\ref{alg:1} only needs $O(n^{1/2})$ iterations in the worst case if the initial estimator $\beta^\dagger$ satisfies $\|\beta^\dagger-\beta^*\|_\infty =O \big(\sqrt{s \log p/n} \big)$, which can be achieved, for example, by using the lasso, where recall that $s$ is the sparsity level of the true signal. 
        
        \begin{algorithm}[h!]
                \begin{enumerate}
                        \item Given training data $y \in \mathbb{R}^n$, $X \in \mathbb{R}^{n \times p}$, and an initial estimator $\beta^\dagger$;
                        \item Pre-process response vector $y$ into a new response vector $\tilde Y = Y - X \beta^\dagger$, so that the new data $(X,\,\tilde Y)$ after the transformation becomes $\tilde Y =  X \tilde\beta^\ast + w$ with $\tilde\beta^\ast = \beta^\ast -\beta^\dagger$;
                        \item Run Algorithm~\ref{alg2} with $(X,\, \tilde Y)$ as the training data and $(X', y'-X' \beta^\dagger)$ as the validation data. Obtain an output $\hat{\beta}$;
                        \item Output $\hat{\beta}+\beta^\dagger$ as the final estimator for $\beta$.
                \end{enumerate}
                \caption{Implicit Regularization for Improving Explicit Regularized Estimators~\label{alg:1}}
        \end{algorithm}

        \subsection{Early Stopping criterion}\label{sec:stopping_rule}
        This subsection discusses several commonly used early stopping rules by treating the iteration number as a tuning parameter. 
        
        \smallskip
        
        \noindent \mbox{\emph{Hold-out or cross-validation}:} The simplest method is to use hold-out data as validation. For example, we can first randomly split the entire data into training data $D_1=(X_1,y_1)$ with size $n_1$ and validation data $D_2=(X_2,y_2)$ with size $n_2$; then run gradient descent on the training data $D_1$ while evaluating the prediction risk $R_{\rm H}(t)=n_2^{-1}\|X_{2}(g_t\circ l_t)-y_{2}\|^2$ on the validation data $D_2$ for all $t  \leq T_{max} $. Since our goal is to minimize the estimation error, whose pattern is consistently captured by $R(t)$, we recommend deciding the final iteration number by:
        \begin{equation}\label{cv}
        \begin{aligned}
        \tilde t &: = \min \{ t:\, t \leq T_{max} \mbox{ and } R_{\rm H}( t+1 ) > R_{\rm H}(t) \} \quad \mbox{or} \quad 
        \tilde t : = {\arg\!\min}_{t \leq T_{max} } R_{\rm H}(t). 
        \end{aligned}
        \end{equation} 
        To make use of the entire dataset, we may perform cross-validation: first split data into $K$ fold; then use one fold as the validation set for evaluating the prediction risk and apply the gradient update to the rest $K-1$ folds; finally, the overall prediction risk $R_{\rm CV}(t)$ at time $t$ is set to be the average prediction error across all $K$ choices of the validation set. Based on $\{R_{\rm CV}(t):\,t\leq T_{max}\}$, criterion~\eqref{cv} can be used to obtain the estimated iteration number $\tilde t$. Finally, we apply this $\tilde t$ from cross-validation to early stop the gradient algorithm with the entire dataset.

        \mbox{\emph{Stein's unbiased risk estimate (SURE)}:} \cite{stein1981estimation} suggested the use of the effective degrees of freedom as a surrogate to the population level prediction risk for selecting a best tuning parameter, which is the stopping time $\tilde t$ in our case. Under our settings, ignoring some higher-order discretization error term of order $O(\eta^2)$ with $\eta$ being the step size, the updating formula of the prediction residual $\tau_t=\big\{X(g_t\circ l_t) - y\big\}\in\mathbb R^n$ induced by the gradient algorithm can be approximated by:
        \begin{equation}\label{eqn:residual}
        \tau_{t+1}\approx \{I - 2 \eta n^{-1} X \mbox{diag}(|g_t\circ l_t|) X^T\} \,\tau_t. 
        \end{equation} 
        Here, diag$(u)$ denotes the diagonal matrix whose diagonals are components of a vector $u$. Let $S_t= \Pi_{s=1}^{t-1} \{I - 2 \eta n^{-1} X \mbox{diag}(|g_s\circ l_s|) X^T\}$ denote the accumulated multiplicative operator up to time $t$,  so that we have the approximation $\tau_t \approx -S_t y$ with $\tau_0\approx -y$. Following~\cite{zou2007degrees}, we define the (approximated) effective degrees of freedom (df) at time $t$ as $\mbox{trace}(S_t)$. According to~\cite{efron2004estimation}, an $C_p$-type statistic for approximating the prediction risk can be related to the effective df via 
        \begin{equation}
        R_{\rm SURE}(t)=\frac{\| \tau_t\|^{2}}{n } + \frac { 2 \mbox{trace}(S_t) } { n } \sigma ^ { 2 }.
        \end{equation} 
        Finally, the stopping time $\tilde t$ can be selected by criterion~\eqref{cv} with $R_{\rm H}(t)$ being replaced by $R_{\rm SURE}(t)$. In practice, we may use any consistent plug-in estimator $\hat{\sigma}^2$, such as the averaged residual sum of squares, to replace the unknown variance $\sigma^2$ in $R_{\rm SURE}(t)$. According to our simulation studies in Section A.3 of the supplementary material, early stopping based on $R_{\rm SURE}(t)$ has a very similar performance as the hold-out or cross-validation method. As a default method, we will use cross validation to select the iteration number in the rest of the paper to avoid estimating $\sigma^2$.

        \section{Theoretical Analysis}\label{Sec:Theory}

        \subsection{Assumptions}
        Recall that the true data generating model is $y=X\beta^*+w$ where $w$ is a vector of independent sub-Gaussian random variables with scales bounded by $\sigma$, and the true parameter $\beta^\ast$ is $s$-sparse.
        Within the $s$ nonzero signal components of $\beta^\ast$, we define the index set of strong signals as $S_1=\{i\in S: |\beta^*_i| \geq  c\sigma \log p \sqrt{{\log p/n}}\}$ and weak signals as $S_2=\{i\in S: |\beta^*_i| \leq  C\sigma \sqrt{{\log p/n}}\}$ for some constants $c,C>0$, where $|S_1|=s_1$, $|S_2|=s_2$. According to the information-theoretic limits from \cite{wainwright2009information}, weak signals of order $\sigma \sqrt{\log p/n}$ in sparse linear regression are generally impossible to be jointly recovered or selected, but they can be detected in terms of the type I/II error control in hypothesis testings \citep{jin2016rare}. Therefore, our primary focus would be the estimation of strong signals.
        We use the notation $\theta_{s_1}(\beta)$ to denote the $s_1$-th largest absolute component value of $\beta$, and let $m=\theta_{s_1}(\beta^\ast)$, which reflects the minimal strength for strong signals. We also use $\kappa$ to denote the strong signal-condition number as the ratio between the largest absolute signal value to the smallest strong signal. 
        We will also make use of the notation of Restricted Isometry Property (\cite{candes2008restricted}), which is a commonly used assumption  \citep{candes2007dantzig} in the high dimensional linear regression literature.
        
        \begin{definition}[Restricted Isometry Property]\label{label:def:1}
                A matrix $X\in \mathbb{R}^{n\times p}$ is said to satisfy the $(s,\delta)$-Restricted Isometry Property if for any $s$-sparse vector $u$ in $\mathbb{R}^{p} $, we have:
                \begin{equation*}
                (1-\delta)\|u\|^2 \leq \frac{1}{n}\,\|Xu\|^2 \leq (1+\delta)\|u\|^2
                \end{equation*}
        \end{definition}
        \noindent As an easy consequence, if matrix $X$ satisfies $(2s,\delta)$-Restricted Isometry Property, then Euclidean inner-product between some sparse vectors is also approximately preserved, that is, $\big|n^{-1}\,\langle Xu, Xv\rangle - \langle u,\,v\rangle \big| \leq \delta\,\|u\|\cdot\|v\|$ holds for any two $s$-sparse vectors $u,\,v\in\mathbb R^p$.

        \noindent With these preparations, we make the following assumptions on the true data generating model, noise $w$, the true parameter $\beta^\ast$, design matrix $X$, initialization parameter $\alpha$ and step size $\eta$ in Algorithm~\ref{alg2}.        
        
        \begin{assumption} \label{(A)}
                The true data generating model is $y=X\beta^*+w$ where $w$ is a vector of independent sub-Gaussian random variables with scales bounded by $\sigma$. For simplicity, we also assume that all $\ell_2$ norm of column vectors of $X$ are normalized to $\sqrt{n}$.
        \end{assumption}
        \begin{assumption} \label{(B)}
                The true parameter $\beta^\ast$ is $s$-sparse, and $s=s_1+s_2$, that is, each nonzero signal in $\beta^\ast$ is either weak or strong. In addition, $\max_{j\in [p]}|\beta_j^\ast| \leq C$ for some constant $C>0$.
        \end{assumption}
        \begin{assumption}  \label{(C)}
                The design matrix $X$ satisfies $(s+1,\delta)$-Restricted Isometry Property with $\delta\lesssim   1 /\{\kappa \sqrt{s}\log(1/\alpha)\}$. 
        \end{assumption}
        \begin{assumption} \label{(D)}
                The initial values for gradient descent are $g_0=\alpha  \+1$, $l_0= 0$, where the initialization parameter $\alpha$ satisfies $0<\alpha\lesssim p^{-1}$, and the step size $\eta$ satisfies $0<\eta\lesssim \{\kappa\log (1/\alpha)\}^{-1}$.
        \end{assumption}
        
        Although our proof heavily relies on the Restricted Isometry Property in Definition~\ref{label:def:1}, the numerical results in Section D of the supplementary materials provide some strong evidence suggesting that the Restricted Isometry Property is not necessary for the method to work. We conjecture that our theoretical conclusion in the following subsection remains valid as long as the much weaker restricted eigenvalue condition~\citep{wainwright2009information,bickel2009simultaneous} is satisfied. More discussions about the consequences due to Assumptions~\ref{(C)} and~\ref{(D)} are provided in the next subsection.

        \subsection{Finite sample error bound}
        Our main theoretical result is summarized in the following theorem.

        \begin{theorem}[Finite sample error under early stopping]\label{thm2}
                Suppose that Assumptions $1-4$ hold.  Then there exist positive constants $( c_1,\, c_2,\, c_3,\, c_4,\, c_5)$ such that for any $M>1$, it holds with probability at least  $1-\exp{(- c_1 \log p)}-\exp{(- c_2\,Ms)}$ that for every time $t$ with $ c_3 \, \log(1/\alpha)/(\eta  m) \leq t\leq  c_4\max\{\log(1/\alpha)/(\eta  m) , 1/(\eta \sigma \sqrt{\log p/n})\} $, the time $t$-iterate $\beta_t=g_t\circ l_t$ satisfies
                \[
                \|\beta_t-\beta^\ast\|^2\leq  c_5 \,\Big(  \alpha^2+ \sigma^2 \frac{  Ms_1}{n}+\sigma^2 \frac{s_2\log p}{n}\Big). 
                \]
        \end{theorem}
        Several remarks are in order. First, the error bound in the theorem suggests that in high-dimensional linear regression, combining early stopping with implicit regularization can significantly improve the estimation accuracy compared to explicitly regularized estimators. For example, when all signals are strong ($s_1=s$ and $s_2=0$), the estimate $\hat \beta =g_t \circ l_t$ attains a parametric rate $\sigma\sqrt{s/n}$ of convergence that is independent of the dimension $p$. In the more general situation where weak signals exist, the overall rate $\sigma \sqrt{s_1/n}+\sigma  \sqrt{s_2\log p/n}$ becomes dependent on the number of weak and strong signals, which is still nearly minimax-optimal \citep{zhao2018pathwise} in this regime. 
        
        Second, according to Theorem~\ref{thm2}, the conditions and the final estimation error bound are not sensitive to the tuning parameters $\eta$ and $\alpha$, as long as they are chosen to be small enough. As a default choice for initialization parameter $\alpha$, we recommend using $\alpha = \min\{n^{-1}, p^{-1}\}$. For the step size $\eta$, we may first select an $\eta_{\rm H}$ based on minimizing the hold-out estimate $R_{\rm H}$ of the prediction risk described in Section~\ref{sec:stopping_rule}, and then set the step size $\eta$ for the entire training data by rescaling $\eta_{\rm H}$ as $\eta = \eta_{\rm H}\sqrt{n_1/n_2}$  {since the choice of $\eta$ according to Theorem~\ref{thm2} satisfies $\eta \lesssim \kappa^{-1} \lesssim \sqrt{\log p /n}$}, where $n_1$ is the training data size and $n_2$ the validation data size in the hold-out. In addition, the range of iteration number $t$ to achieve the desired estimation accuracy is wide, making the empirical performance less sensitive to the choice of the early stopping criterion.
        
        Third, the proof of Theorem~\ref{thm2} is based on a similar strategy as in \cite{li2018algorithmic} for analyzing matrix factorized gradient descent in noiseless matrix sensing by dividing the gradient dynamic into two stages, where in the first stage, all significant signals stand out by quickly moving away from zero, as shown in Proposition
        1 in Section C.3 in the supplement. In the second stage, all insignificant signals remain bounded and small while significant ones exponentially converge to their true values modulo statistical errors, as proved in
        Proposition 2 in Section C.3 in the supplement. 
        To facilitate the analysis, in our proofs we adopt another reparametrization of $g=a+b$ and $\ell = a-b$ so that $\beta = a^2-b^2 = (a-b)\circ (a+b)$ to keep track of
        the change of signs in components of parameter $\beta$. Interestingly, after the first version of our work was posted on arXiv, \cite{vaskevicius2019implicit} proposed a similar implicit regularization algorithm for linear regression directly based on this parametrization. They presented a significantly improved analysis by proving an estimation error bound under weaker conditions on restricted isometry constant, step size, and strength of signals, and condition number $\kappa$ than ours. However, their error bound is slightly worse than ours with an extra $\log s_1$ factor in front of the $\sigma^2 M s_1/n$ term corresponding to strong signals.
        
        As a final remark, Theorem~\ref{thm2} also reflects some algorithmic benefit of Algorithm~\ref{alg:1} based on a consistent initial estimator $\beta^\dagger$ over Algorithm~\ref{alg2}. Specifically, consider the largest signal conditional number $\kappa$ regime where $\kappa \approx \sqrt{n/\log p}$ , which corresponds to the worse case scenario where the minimal signal strength $m$ is of order $\sqrt{\log p/n}$, i.e.~the ratio $\kappa$ between the largest component in absolute value in $\beta^\ast$ and the smallest component is large. In this extreme setting, Assumption~\ref{(D)} requires $\eta \lesssim (\kappa \log p)^{-1}=O(n^{-1/2})$. As a consequence, Theorem~\ref{thm2} suggests that at least $t \gtrsim \kappa \log p /\eta = O(n)$ iterations is needed to guarantee Algorithm~\ref{alg2} to achieve the desired error bound. 
        In comparison, suppose the initial estimator $\beta^\dagger$ in Algorithm~\ref{alg:1} satisfies some error bound condition as $\|\beta^\dagger -\beta^\ast\|_\infty \leq \xi_n$, then after the pre-processing step 2, the new conditional number $\tilde\kappa$ can always be upper bounded by $\xi_n\sqrt{n/\log p}$, where the worst case occurs when $m\approx \sqrt{\log p/n}$, so that now only $\kappa\log p /\eta = O(\xi_nn)$ iterations is needed for Algorithm~\ref{alg2}. For example, if $\|\beta^\dagger -\beta^\ast\|_\infty \leq \xi_n= O_p(\sqrt{s\log p/n})$, which is satisfied by estimators from the lasso, then the iteration number can be reduced from $O(n)$ to $O(\sqrt{sn})$. This iteration complexity leads to a total computational complexity of $O(s^{1/2}n^{3/2})$,  which has the same complexity scaling in terms of $n$ as the improved algorithm proposed in~\cite{vaskevicius2019implicit}.

        \section{Simulations and Real Data Analysis}\label{sec:simu}
        
        We further demonstrate the advantages of our algorithm by considering simulation settings S1-S8 with  $X\in\mathbb R^{n\times p}$, rows $X^{(i)}\overset{iid}{\sim}\mathcal N(0,\,\Sigma)\in \mathbb R^p$ and the sparsity level $s=4$. For S1-S4, we set $n=200$, $p=500$, $\Sigma_{jk}=\rho^{|j-k|}$ with $\rho =0,0.1,0.2,0.5$ (here we denote $0^0 =1$), while S5-S6 corresponds to S1-S4 with $p=2000$ but everything else the same.  For the case of strong signals we set signals as $-1,2,2,3$ and noise level $\sigma$ with $\sigma =0.15*\|\beta^\ast\|$; while for weak-signals case, we set $\sigma=1$ and the four signals are all $2\sqrt{\log p/n}$. We generate $3 n$ observations independently and split into $3$ even parts, then use the first part for training, the second part for validation and the final part for testing. The evaluation metric is standardized estimation error $\|\widehat \beta-\beta^\ast\|^2/\|\beta^\ast\|^2$ and mean prediction error $\sqrt{\|y-\hat{y}\|^2/n}$  for the test data set. We compare the median of the standardized estimation errors and prediction errors with the lasso, smoothly clipped absolute deviation and minimax concave penalty by repeating $50$ times. We implement the lasso using \textit{glmnet} R package \citep{friedman2010regularization} while for smoothly clipped absolute deviation and minimax concave penalty, we use the R package $\textit{ncvreg}$ \citep{breheny2011coordinate}. The standard error of medians are calculated by bootstrapping the calculated errors $1000$ times. For our algorithm, we use the initialization $\alpha=10^{-5}$. For all other methods, the hyperparameter $\lambda$ is chosen based on the prediction performance on the validation data from $10000$ evenly spaced grids within $[0.001,1]$. The simulation results in Table~\ref{table1} below and
        Table 2 in Section A.6 in the supplement indicate that our methods consistently have the best performance
        over all explicit penalization-based competitors across all settings.

        \begin{table}[h]        \caption{Simulation result for median of standardized estimation error\label{table1}
                }
                \footnotesize
                \begin{tabular*}{\textwidth}{c@{\extracolsep{\fill}}c@{\extracolsep{\fill}}cccccccc}

                        &       &S1 &S2 &S3 &S4 &S5 &S6 &S7 & S8 \\
                        
                        \multirow{1}{3em}{Strong Signals}                       &\multirow{2}{3em}{GD}  
                        &{0.520} &{0.448} &{0.510} &{0.568} &{0.385} &{0.290} &{0.465}         &{0.460}
                        
                        \\
                        &               &(0.0428)       &(0.0530)       &(0.0607)       &(0.0850)       &(0.0533)       &(0.0465)       &(0.0858)       &(0.0863) 
                        \\
                        
                        &\multirow{2}{3em}{lasso} 
                        &3.11   &3.10   &3.42   &4.51   &4.62   &4.04   &4.40   &6.98
                        \\
                        &               &(0.219)        &(0.173)        &(0.242)        &(0.274)        &(0.279)        &(0.205)        &(0.306)        &(0.452)
                        \\
                        
                        &\multirow{2}{3em}{SCAD}
                        &0.613  &0.533  &0.650  &0.691  &0.519  &0.401  &0.595  &0.646
                        \\
                        &               &(0.0464)       &(0.0679)       &(0.0702)       &(0.103)        &(0.0527)       &(0.0574)       &(0.0837)       &(0.0776)
                        \\
                        
                        &\multirow{2}{3em}{MCP}
                        &0.628  &0.552  &0.594  &0.733  &0.484  &0.405  &0.595  &0.708
                        \\
                        &               &(0.0392)       &(0.0779)       &(0.0809)       &(0.0902)       &(0.0706)       &(0.0597)       &(0.0741)       &(0.0680)
                        \\
                        
                        \multirow{1}{4em}{Weak Signals} 
                        &\multirow{2}{3em}{GD}  
                        &{0.996} &{1.226} &{0.586} &{0.651} &{0.699} &{0.474} &{0.556}         &{0.368}
                        
                        \\
                        &       &(0.259)        &(0.255)        &(0.0938)       &(0.122)        &(0.117)        &(0.0609)       &(0.127)        &(0.0351) 
                        \\
                        
                        &\multirow{2}{3em}{lasso} 
                        &2.530  &3.063  &2.292  &2.620  &2.048  &1.952  &1.330  &1.178
                        \\
                        &       &(0.195)        &(0.151)        &(0.119)        &(0.168)        &(0.208)        &(0.121)        &(0.100)        &(0.115)
                        \\
                        
                        &\multirow{2}{3em}{SCAD}
                        &3.942  &3.781  &3.331  &3.204  &3.241  &2.369  &3.881  &3.397
                        \\
                        &       &(0.409)        &(0.137)        &(0.197)        &(0.192)        &(0.219)        &(0.150)        &(0.270)        &(0.313)
                        \\
                        
                        &\multirow{2}{3em}{MCP}
                        &2.629  &2.153  &2.248  &1.730  &2.878  &1.570  &6.149  &4.984
                        \\
                        &       &(0.279)        &(0.180)        &(0.261)        &(0.135)        &(0.346)        &(0.174)        &(0.312)        &(0.447)
                        \\
                        
                \end{tabular*}
                \begin{tablenotes}
                        \small
                        \item GD, gradient descent; MCP, minimax concave penalty; SCAD, smoothly clipped absolute deviation.  For each method, the standard derivation is in the parenthesis under the median. All the numbers under the strong
                        (weak) signals have been multiplied by $10^3(10^1)$.
                \end{tablenotes}
                
        \end{table}

        Now we consider variable selection when there exist some weak signals. Suppose the simulation settings are similar to S3 above, with only the true signals change. Let $s=20$, and the strength of first $4$ signals is $0.5 \sigma \sqrt{\log p/n}$, while the other $16$ are $ 5 \sigma \sqrt{ \log p / n}$, where $\sigma =1$. The first $4$ signals are too weak to be selected by all methods. However, since all methods are based on minimizing the prediction error, the effect of these weak signals pertains, and may increase the false discovery rate. Under the above settings, we perform a model selection based on minimized prediction errors through $5$-fold cross-validation. For our method, we use the same regularization parameter as the lasso to perform a hard threshold after estimation. We repeat the process $50$ times and compare variable selection errors. From panel (a) in Figure~\ref{fi:box}, we can see our method are robust to the enhancement of false detection due to failure on detecting weak signals: although the true negative error of our method is $4$, which means all weak signals can not be detected, the false detections of our methods are closed to zero. Minimax concave penalty performs similarly with our method in true negative but slightly worse in false positive, while for the lasso and smoothly clipped absolute deviation,  although sometimes weak signals can be detected, the risk of false detections is high. Overall, our methods perform a consistent variable selection for strong signals and achieve better estimation than the competitors.

        \begin{figure}
                \centering
                \subfigure[Variable selection results for simulations]{
                        \includegraphics[width=0.49\linewidth]{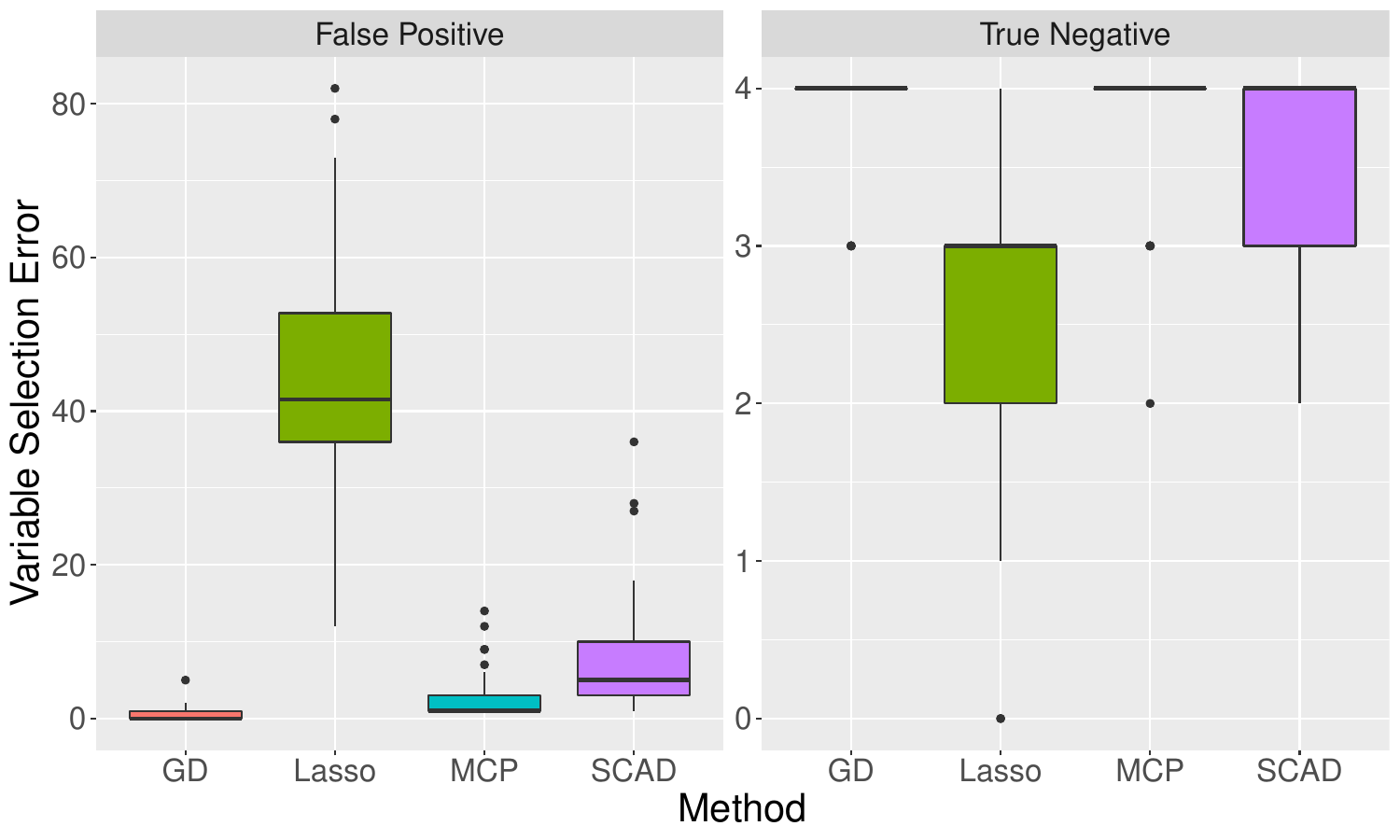}}
                \subfigure[Prediction errors for Riboflavin data set]{
                        \includegraphics[width=0.49\linewidth]{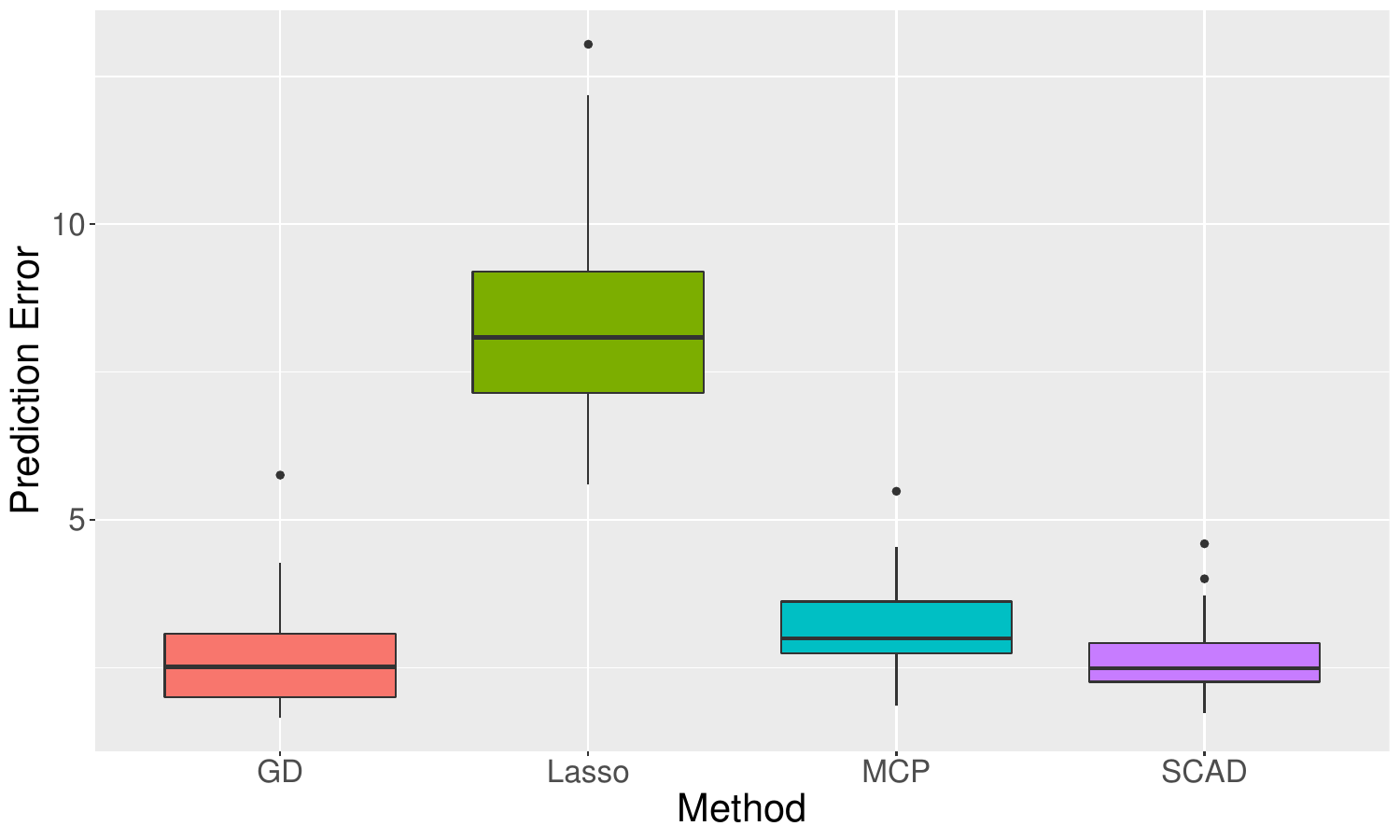}}  
                \caption{Panel (a) is the variable selection errors for selected model based on minimized prediction cross validation errors.  `False positive' means the truth is zero but detected as signal;  `true negative' means the truth is nonzero but not detected. Panel (b) is the prediction errors on the test data of Riboflavin data set for each method. $x$-axis stands for the methods used for estimation, and $y$-axis stands for the testing prediction error $\|y-\hat{y}\|$. GD, gradient descent; MCP, minimax concave penalty; SCAD, smoothly clipped absolute deviation.}\label{fi:box}
        \end{figure}

        We compare our method with others to analyze the Riboflavin data set \citep{buhlmann2014high}, which is available in \textit{hdi} R package. The dataset contains $71$ observations of log-transformed riboflavin production rate versus the logarithm of the expression level of 4088 genes. Before estimation, we first perform independence screening \citep{fan2008sure} based on the rank of the correlation strength for each predictor versus response to decrease the dimension of feature space into $500$. Then we normalize and add the intercept column into the design matrix. For evaluation, we split the observations into $50$ training samples and $21$ testing samples, with performing $10$-fold cross validation to select iteration steps and regularization parameters in the training data. Still, for our algorithm, we use the initial value $\alpha=10^{-5} $ for all training processes. We record the prediction errors for the testing data set and repeat $50$ times. From panel (b) in Fig.~\ref{fi:box}, our method and smoothly clipped absolute deviation obtain the least prediction errors.

        \section{Discussion}
        We discuss several important open problems on implicit regularization as our future directions. First, our theory heavily relies on the RIP of the design matrix, which is relatively strong comparing to the restricted eigenvalue condition \citep{bickel2009simultaneous} as the minimal possible assumption in the literature. It would be interesting to theoretically investigate whether our results remain valid without the RIP. Second, it is of practical importance to study whether any computationally-efficient early stopping rule based on certain data-driven model complexity measures rather than the cross validation method can be applied to reliably and robustly select the tuning parameters, such as the iteration number and the step size in our algorithm.         
        \section*{Acknowledgment}       
        Dr. Zhao's work was partially supported by National Science Foundation grant CCF-1934904. Dr. Yang's work was partially supported by National Science Foundation grant DMS-1810831. We would like to thank the editor, associate editor and two anonymous reviewers from Biometrika for their careful comments and helpful suggestions that significantly improved the quality of the paper.        
        \section*{Supplementary material}
        Supplementary material contains detailed discussions on dynamic system interpretation of the implicit regularization, extension to adaptive step size and elastic net, proof of the theorems, extra simulations results beyond design with RIP, some reproducible simulation codes and a detailed comparison with \cite{vaskevicius2019implicit}.
        
        \appendix 
        
        	\section{Some additional results}   
        \subsection{Dynamical system interpretation for section~\ref{Sec:Heuristic}}
        To emphasize the dependence of the solution on $\alpha$, we instead write $g(t),\,l(t),\,r(t)$ as $g(t,\alpha), \,l(t,\alpha),\,r(t,\alpha)$.
        For illustration purposes, we assume that the limiting point of this system is continuous and bounded as the initialization value $\alpha\to 0_+$, that is, both limits $g_\infty=\lim_{t\to\infty, \alpha\to 0_+}g(t,\alpha)$ and $l_\infty=\lim_{t\to\infty, \alpha\to 0_+}l(t,\alpha)$ exist in $\mathbb R^p$ and are finite. 
        
        Let $s(t,\alpha)=\int_0^t r(\tau,\alpha) d\tau\in\mathbb R^p$, then simple calculation leads to the relation
        \begin{equation*}
        \left[ \begin{array}{c} g_j(t,\alpha)+l_j(t,\alpha) \\[0.3em] g_j(t,\alpha)-l_j(t,\alpha) \end{array} \right] 
        = \alpha\,\left[ \begin{array}{c}\exp(-X_j^{T}s(t,\alpha)) \\[0.3em]  \exp(X_j^{T}s(t,\alpha)) \end{array} \right],\quad\mbox{for each }j=1,\ldots,p.
        \end{equation*}
        Under the aforementioned assumption on the existence of limits as $t\to\infty$ and $\alpha\to 0_+$, the preceding display implies one of the following three situations for each $j$:
        \begin{equation*}
        \begin{aligned}
        &\mbox{[Case 1:]} &g_{j,\infty}=l_{j,\infty}\neq 0, \mbox{ and} 
        &\displaystyle \lim_{t\to\infty,\alpha\to 0_+}X_j^{T} s(t,\alpha)/\log(\alpha) =  1. \\
        &\mbox{[Case 2:]} &g_{j,\infty}=-l_{j,\infty}\neq 0, \mbox{ and} 
        &\displaystyle \lim_{t\to\infty,\alpha\to 0_+}X_j^{T} s(t,\alpha)/\log(\alpha) =  -1. \\
        &\mbox{[Case 3:]} &g_{j,\infty}=l_{j,\infty}=0, \mbox{ and} 
        &\displaystyle \lim_{t\to\infty,\alpha\to 0_+}X_j^{T} s(t,\alpha)/\log(\alpha) =\gamma_j\in[-1,1].
        \end{aligned}
        \end{equation*}
        Denote $s_\infty$ as the limit $\lim_{t\to\infty,\alpha\to 0_+} s(t,\alpha)/\log(\alpha)$. Recall $\beta_\infty = g_\infty\circ l_\infty$, and the previous three cases can be unified into
        $$
        X_j^{T}s_\infty=
        \begin{cases} 
        \mbox{sign}(\beta_{j,\infty}), & \mbox{if}\  \beta_{j,\infty}\neq 0, \\ 
        \gamma_j\in [-1,1], &  \mbox{if}\ \beta_{j,\infty}= 0,
        \end{cases}\quad\mbox{for each }j=1,\ldots,p.
        $$
        This identity, together with the limiting point condition $X\beta_{\infty}=y$
        coincides with the Karush-Kuhn-Tucker condition for the basis pursuit problem~\eqref{bs}.

        Under the same aforementioned assumption, if there is additive noise on $y$, then based on the integral mean value theorem, the preceding display implies one of the following three situations for each $j$:
        \begin{equation*}
        \begin{aligned}
        &\mbox{[Case 1:]} &g_{j,\infty}=l_{j,\infty}\neq 0, \mbox{ and} 
        &\displaystyle \lim_{t\to\infty,\alpha\to 0_+}X_j^{T}r(t,\alpha)=  \log(1/\alpha)/t. \\
        &\mbox{[Case 2:]} &g_{j,\infty}=-l_{j,\infty}\neq 0, \mbox{ and}  
        &\displaystyle \lim_{t\to\infty,\alpha\to 0_+}X_j^{T} r(t,\alpha)=  -\log(1/\alpha)/t. \\
        &\mbox{[Case 3:]} &g_{j,\infty}=l_{j,\infty}=0, \mbox{ and} 
        &\displaystyle \lim_{t\to\infty,\alpha\to 0_+}X_j^{T} r(t,\alpha) t/\log(1/\alpha) =\gamma_j\in[-1,1].
        \end{aligned}
        \end{equation*}

        \subsection{Adaptive step size and variable selection} \label{sec:2.5}
        
        A nature extension of gradient descent is to assign different weights (step sizes) to different coordinates of $\beta$, which is related to the adaptive Lasso \citep{zou2006adaptive}. It can be seen from the differential equation interpretation: by inserting a constant weighting matrix $D(\Omega)=\mbox{diag}(\omega_1,...,\omega_p)$ into the equation~\eqref{de} in the main manuscript, we obtain the limiting dynamical system as
        \begin{equation*}
        \begin{cases} 
        \ \dot{g}(t)=-\big[D(\Omega) X^{T}r(t)\big] \circ l(t),\\ 
        \ \, \dot{l}(t)=-\big[D(\Omega)X^{T}r(t)\big] \circ g(t).
        \end{cases}
        \end{equation*}
        Based on similar heuristic analysis as in Section~\ref{Sec:Heuristic} for the noiseless case, the limiting point of the dynamic system satisfies:
        \begin{equation*}
        X_j^{T}s_\infty=
        \begin{cases} 
        \mbox{sign}(\beta_{j,\infty}) /\omega_j, & \mbox{if}\  \beta_{j,\infty}\neq 0, \\ 
        \gamma_j\in [-\frac{1}{\omega_j},\frac{1}{\omega_j}], &  \mbox{if}\ \beta_{j,\infty}= 0,
        \end{cases}\quad\mbox{for each }j=1,\ldots,p.
        \end{equation*}
        which is the KKT condition for the dual form of the adaptive Lasso 
        \begin{equation*}
        \min_{\beta\in\mathbb R^p}\sum_{j=1}^p\frac{|\beta_j|}{w_j} \qquad\mbox{subject to } X\beta =y.
        \end{equation*}
        In the limiting case when the step size $\omega_j$ of a particular component $\beta_j$ tends to $0$, we are equivalently adding an $+\infty$ when $\beta_j\neq 0$. In contrast, if we apply a larger step size $\omega_j$ to $\beta_j$, then $\beta_j=g_j\circ l_j$ tends to move faster and more freely in the parameter space, which is equivalent to a smaller penalty on $\beta_j$.
        The original paper in \cite{zou2006adaptive} constructed the weights based on the ordinary least square solution, which requires $n\geq p$. In practice, when $p>n$, we can construct weights through a preprocessing step. For example, variable screening can be applied to introduce sparse weights.

        To enable variable selection in our method, we can perform a component-wise hard thresholding operation to the final estimator $\hat{\beta}=g_{{\tilde t}}\circ l_{{\tilde t}}$. Based on our theoretical analysis, since our method tries to shrink both weak signals and errors into very small order $p^{-2}$, it is more robust to false detections than other explicit regularizations when the same tuning parameter for noise level is applied. Let us consider a simple example to illustrate the basic idea: we set $n=10$, $p=20$, $X_{ij} \overset{iid}{\sim} \mathbb{ N }(0,1)$ for $i=1,2,...,n$ and $j=1,2,...,p$,  $\beta^*_1=0.5 \sigma\sqrt{\log p/n}$, $\beta^*_2=5 \sigma \sqrt{\log p/n}$, and all other components are zeros in the data generating model $y=X\beta^* +w$ with $w \sim \mathbb{ N }(0,I)$.  Since the strength of the first components of truth is weak, it is hard to be detected by all methods we have tried. However, the effect of the weak signals on $y$ still pertains. In particular, when applying cross-validation, traditional penalized methods tend to over-select the predictors, leading to many false discoveries.  In comparison, due to the implicit regularization, our method tends to be more robust to the weak signals---our estimate is typically non-sparse. The effect of the non-detected weak signals can be distributed to all components of the estimated vector, and no component is particularly spiked. Consequently, our method tends to be more robust to false discoveries after applying the hard thresholding. The variable selection results are shown in Figure~\ref{fi:vsl}. As we can see, the Lasso can not detect the weak signal, and two components, indexed by $6,19$, appear to be falsely detected through cross-validation (note that in Lasso, soft thresholding has already been applied). In contrast, in our method, most unimportant signals remain small. Performing hard thresholding with the same regularization parameter selected by the Lasso can erase all false detections, leading to the selection of strong signals only.  
        \begin{figure}[t]
        	\includegraphics[width=\linewidth]{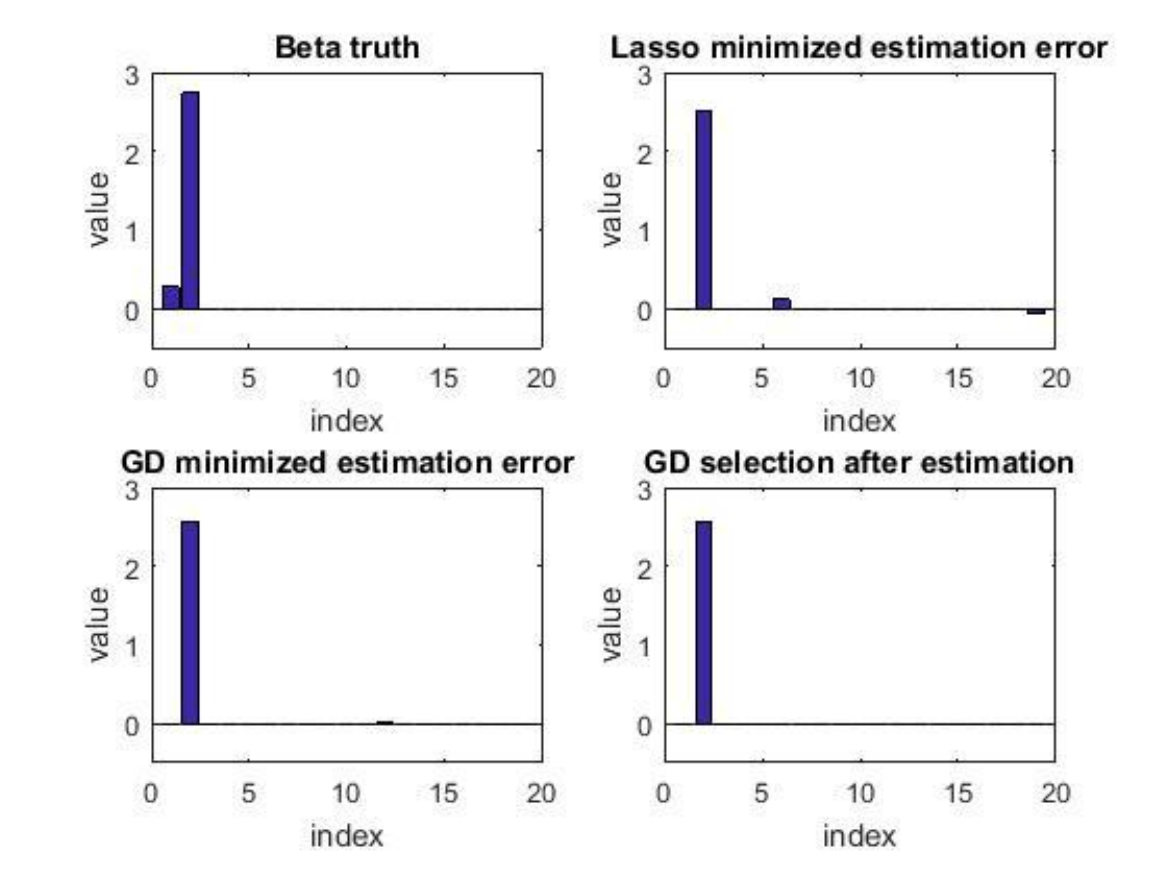}
        	\caption{The values versus index for truth $\beta^*$, $\beta$ estimator through lasso by minimizing the cross validation error in prediction, $\beta$ estimator through gradient descent by minimizing the cross validation error in prediction and $\beta$ estimator through `post estimation' selection for gradient descent.} \label{fi:vsl}
        \end{figure}
        
        \subsection{Empirical Comparisons between different early stopping criteria}
        We adopt the simulation framework S1-S4 for the strong-signal setting (only change the standard derivation to $\sigma=0.1$) in section~\ref{sec:simu} to compare different early stopping criteria. We record the mean estimation errors averaging over $50$ trials and report the errors in figure~\ref{fig:33}.
        \begin{figure}[h]
        	\centering
        	\subfigure[Comparisons in stopping criteria]{\includegraphics[scale=0.4]{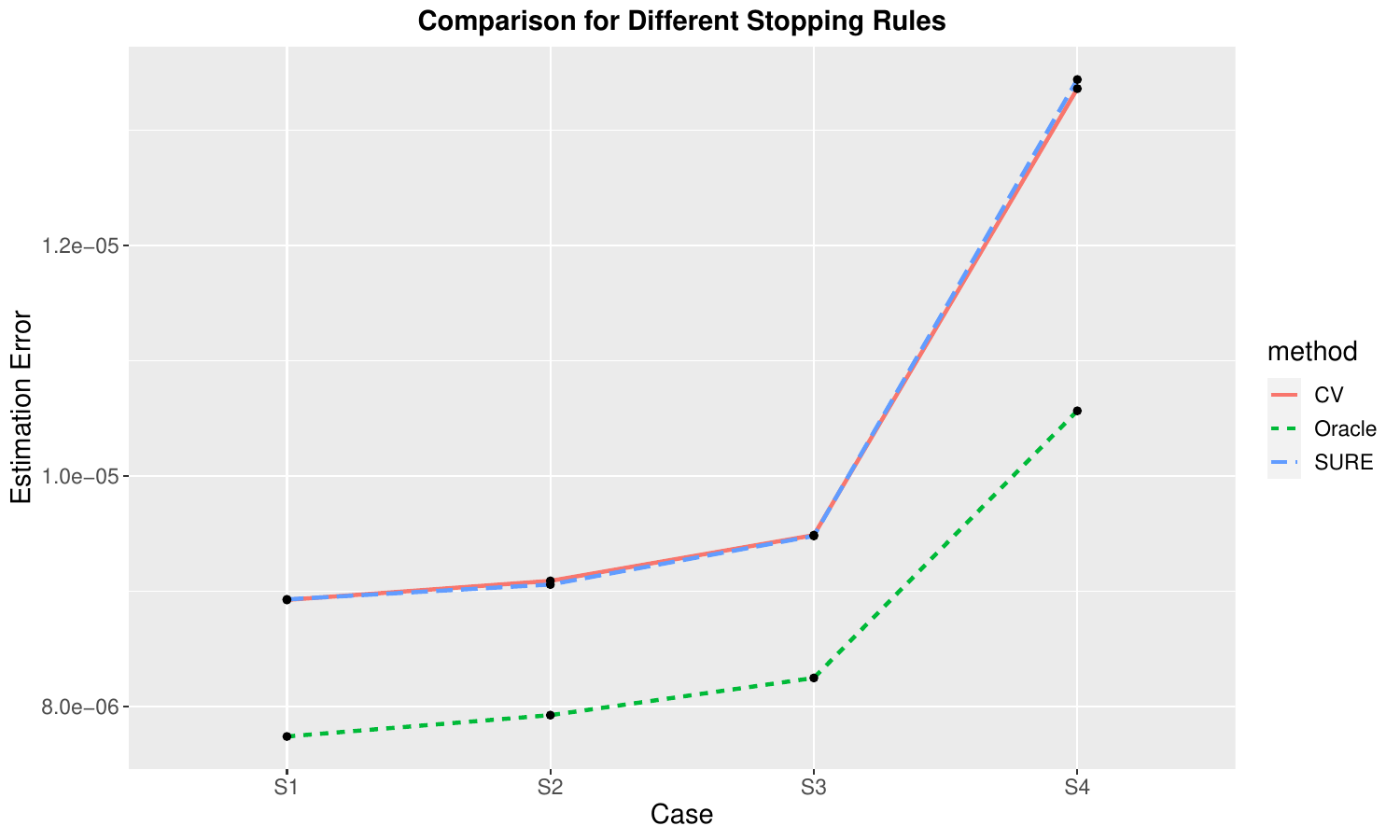}}   	
        	\caption{The comparison of the estimation errors for different early stopping rules, `Oracle' stands for the optimal early stopping rule with knowledge on the truth. 'CV' stand for early stopping through $5$ fold cross validation. S1-S4 are simulation frameworks adopted in section $4$.
        	} \label{fig:33}
        \end{figure}   
        
        \subsection{Extension to Elastic Net}
        
        When both $\ell_1$ and $\ell_2$ type of regularizations are of interest to be applied,  we can use the same Hadamard reparametrization $\beta = g\circ l$, and the following objective function
        \begin{equation*}
        f(g,l) := \frac{1}{2n} \|X(g\circ l)-y\|^2 +\frac{\lambda}{2} \|g \circ l\|^2,
        \end{equation*}
        where $\lambda$ controls the trade-off between $\ell_1$ and $\ell_2$ norms. The gradient of the problem is:
        \begin{equation*}
        \begin{aligned}
        \nabla f_g = l \circ  [ (X^TX/n+\lambda I )g \circ l -X^Ty] \\
        \nabla f_l = g \circ  [ (X^TX/n+\lambda I )g \circ l -X^Ty]. 
        \end{aligned}
        \end{equation*}
        
        In addition, denote $R=(X^TX/n+\lambda I )g \circ l -X^Ty$, then the Hessian matrix is
        \begin{align*}
        \begin{pmatrix}
        \mbox{Diag}(l) & 0\\
        0 & \mbox{Diag}(g)
        \end{pmatrix}
        \begin{pmatrix}
        n^{-1}X^TX+\lambda I & n^{-1}X^TX +\lambda I\\
        n^{-1}X^TX +\lambda I & n^{-1}X^TX +\lambda I
        \end{pmatrix}
        \begin{pmatrix}
        \mbox{Diag}(l) & 0\\
        0 & \mbox{Diag}(g)
        \end{pmatrix}
        \\ +\ \
        \begin{pmatrix}
        0 & \mbox{Diag}(R)\\
        \mbox{Diag}(R) & 0
        \end{pmatrix}.
        \end{align*}
        Since ridge regression parameterized by $\beta$ is convex, by using the same proof technique with lemma 1, we can obtain that $f(g; l)$ does not have local maximums, all its local minimums are the global minimum, and all saddle points are strict saddle. In particular, $(\bar g, \bar l)$ is a global minimum of $f(g; l)$ if and only if
        $$ (X^TX/n+\lambda I ) \bar g \circ \bar l -X^Ty = 0,$$ which reflects the effect of the $\ell_2$ penalty.
        
        In addition, since we have the structure $\nabla f_g = l \circ X^T r$ and $\nabla f_l = g \circ X^T r$, where $r=(X/n +\lambda X^{-} g\circ l) -y $ with $X^{-}$ as the pseudo inverse of $X$, we can adopt the dynamic system formula as follows:
        \begin{equation}
        \begin{cases} 
        \ \dot{g}(t)=-\big[ X^{T}r(t)\big] \circ l(t),\\ 
        \ \, \dot{l}(t)=-\big[X^{T}r(t)\big] \circ g(t),
        \end{cases}\quad\mbox{with initialization}\quad 
        \begin{cases} 
        \ g(0)=\alpha  1,\\ 
        \ \, l(0)= 0,
        \end{cases}
        \end{equation}
        Then we  have
        $$
        X_j^{T}s_\infty=
        \begin{cases} 
        \mbox{sign}(\beta_{j,\infty}), & \mbox{if}\  \beta_{j,\infty}\neq 0, \\ 
        \gamma_j\in [-1,1], &  \mbox{if}\ \beta_{j,\infty}= 0,
        \end{cases}\quad\mbox{for each }j=1,\ldots,p.
        $$
        where $s_\infty = \lim_{t\to\infty,\alpha\to 0_+} \int_0^t r(\tau,\alpha) d\tau/\log(\alpha)$, which reflects the effect of the $\ell_1$ penalty.
        
        Therefore, the above formulation for $f(g,l)$ can be seen as applying implicit regularization for the elastic net.

        \subsection{Choosing the step size and initial value}
        According to the main theorem, the parameters $\eta$ and $\alpha$  are not sensitive to the final rate of estimation errors,  as long as we choose  $\eta$ and $\alpha$ to be small enough, the optimal error rate in estimation can be obtained with properly tuned stopping time. 
        
        For $\alpha$, we can simply set $\alpha = 1/p$ so that $\alpha^2$ will be term of a smaller order in the estimation error according to the main theorem. For choosing suitable $\eta$, we can first choose $\eta_1$ based on a small subset of the whole training data with observation $n_1$. Note that we require $\eta \lesssim \kappa^{-1} \lesssim \sqrt{\log p /n}$, so we can obtain $\eta = \eta_1\sqrt{n_1/n}$.
        
        \subsection{Simulation result for median of mean prediction error}
        \begin{table}[h]        \caption{Simulation result for median of mean prediction error \label{table2}
        	}
        	\footnotesize
        	\begin{threeparttable}
        		\begin{tabular*}{\textwidth}{c@{\extracolsep{\fill}}c@{\extracolsep{\fill}}cccccccc}
        			
        			&       &S1 &S2 &S3 &S4 &S5 &S6 &S7 & S8 \\
        			
        			\multirow{1}{4em}{Strong Signals}               &       \multirow{2}{3em}{GD}  
        			&0.638  & {0.636} &0.640  & {0.634} & {0.646} & {0.651} & {0.641}         & {0.642}
        			
        			\\
        			&               &(0.0597)       &(0.0753)       &(0.0718)       &(0.0709)       &(0.0491)       &(0.0510)       &(0.0406)       &(0.0498) 
        			\\
        			
        			&               \multirow{2}{3em}{lasso} 
        			&0.676  &0.672  &0.671  &0.685  &0.693  &0.699  &0.696  &0.708
        			\\
        			&               &(0.0899)       &(0.0981)       &(0.0932)       &(0.0510)       &(0.0568)       &(0.0918)       &(0.0615)       &(0.0488)
        			\\
        			
        			&               \multirow{2}{3em}{SCAD}
        			&0.638  &0.638  & {0.637} &0.637  &0.650  &0.654  &0.643  &0.647
        			\\
        			&               &(0.0530)       &(0.0724)       &(0.0716)       &(0.0713)       &(0.0470)       &(0.0451)       &(0.0402)       &(0.0434)
        			\\
        			
        			&               \multirow{2}{3em}{MCP}
        			&\textbf{0.637} &0.639  &0.638  &0.637  &0.650  &0.652  &0.644  &0.647
        			\\
        			&               &(0.0516)       &(0.0756)       &(0.0684)       &(0.0753)       &(0.0475)       &(0.0441)       &(0.0435)       &(0.0408)
        			\\
        			
        			\multirow{1}{3em}{Weak Signals}         &               \multirow{2}{3em}{GD}  
        			& {1.023} & {1.039} & {1.007} & {1.038} & {1.013} & {1.030} & {1.000}         & {1.013}
        			
        			\\
        			&               &(0.114)        &(0.128)        &(0.141)        &(0.111)        &(0.132)        &(0.089)        &(0.151)        &(0.069) 
        			\\
        			
        			&               \multirow{2}{3em}{lasso} 
        			&1.061  &1.093  &1.050  &1.097  &1.055  &1.070  &1.034  &1.054
        			\\
        			&               &(0.0908)       &(0.0841)       &(0.138)        &(0.0924)       &(0.0921)       &(0.0873)       &(0.132)        &(0.106)
        			\\
        			
        			&               \multirow{2}{3em}{SCAD}
        			&1.082  &1.108  &1.066  &1.110  &1.074  &1.091  &1.070  &1.094
        			\\
        			&               &(0.107)        &(0.104)        &(0.197)        &(0.0868)       &(0.0866)       &(0.0914)       &(0.141)        &(0.110)
        			\\
        			
        			&               \multirow{2}{3em}{MCP}
        			&1.057  &1.067  &1.038  &1.071  &1.055  &1.063  &1.066  &1.092
        			\\
        			&               &(0.0953)       &(0.0810)       &(0.162)        &(0.112)        &(0.115)        &(0.0961)       &(0.139)        &(0.0640)
        			\\ 
        		\end{tabular*}
        		\begin{tablenotes}
        			\small
        			\item GD, gradient descent; MCP, minimax concave penalty; SCAD, smoothly clipped absolute deviation.  For each method, the standard derivation is in the parenthesis under the median. There are $10^{-1}$ factors for all standard derivations.
        		\end{tablenotes}
        	\end{threeparttable}
        \end{table}

        \subsection{Simulations on the effect of error variance on the optimal stopping region}

        Let the truth be $(1,-1,1,-1,0,...,0)$, $X_{ij} \sim N(0,1)$, $\alpha=10^{-10}$  and $\sigma=0.01,0.1,0.5,1$. Note that we have $\kappa=1$ and $m=1$, by fixing the initial value $\alpha$ and $\eta$, we want to see the effect of the variance $\sigma$ on the  optimal stabilized region of the stopping time:
        \begin{figure}[h]
        	\centering
        	\subfigure[Comparisons in stopping criteria]{\includegraphics[scale=0.7]{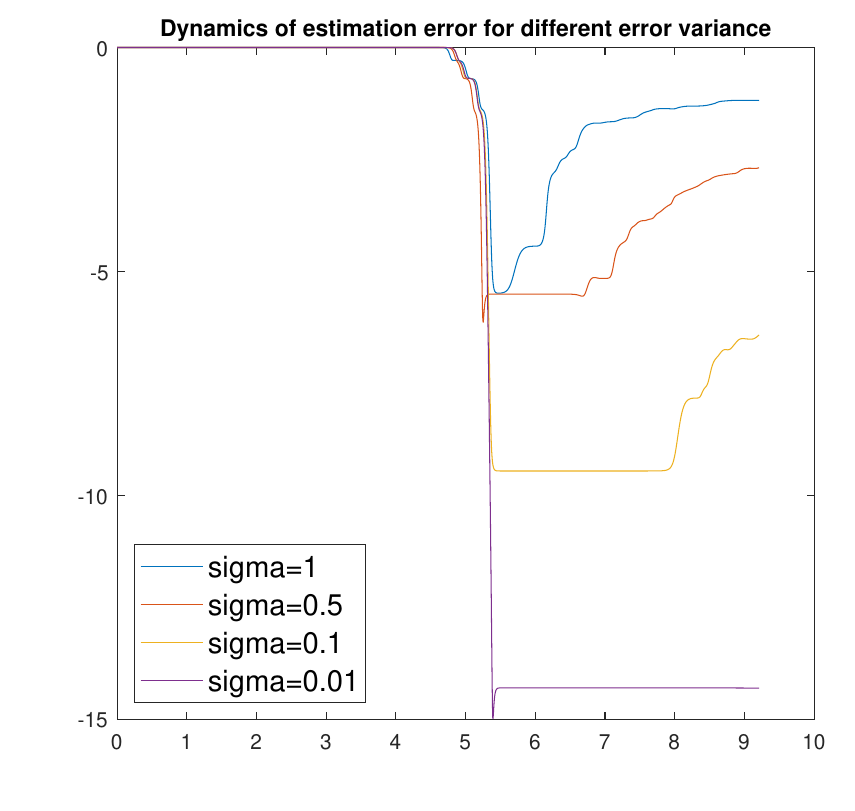}}   	
        	\caption{The log-log plot of standardized estimation error $\|\widehat \beta-\beta^\ast\|^2/\|\beta^\ast\|^2$ versus iteration number $t  $ for gradient descent.
        	} \label{fig:region}
        \end{figure}
        
        The above figure perfectly matches our theory: the start point of the optimal stopping time is at the same rate when the initial values, conditional numbers, and the smallest signals are the same; while the endpoint of the optimal stopping time is also affected by the error variance $\sigma$. While $\sigma$ is becoming smaller, the region of the optimal stopping time will be more expanded.
        
        \section{Proof of the results in the paper}

        \subsection{Notation}
        Recall that $\|v\|=\sqrt{\sum_{j=1}^p v_j^2}$ and $\|v\|_\infty=\max_{j}|v_j|$ denote the vector $\ell_2$-norm and $\ell_\infty$-norm, respectively. Moreover, $I$ is the identity matrix in $\mathbb R^p$, and for any subset $S$ of $\{1,\ldots,p\}$, $I_S$ is the diagonal matrix with $1$ on the $j$th diagonals for $j\in S$ and $0$ elsewhere. We use bold letter $\mathbf 1\in\mathbb R^p$ to denote an all-one vector. $\theta_s(\beta)$ denote the $s$-largest component of vector $\beta\in\mathbb R^p$ in absolute value. We use notation $\lesssim$ and $\gtrsim$ to denote $\leq$ and $\geq$ up to some positive multiplicative constant, respectively. For two vectors $u$ and $v$ of the same dimension, we use $a\geq b$ and $a\leq b$ to denote element-wise $\geq$ and $\leq$. Denote $\lambda_{\max}(A)$ and $\lambda_{\min}(A)$ be the maximal and minimal eigenvalues of matrix $A$. Through this document, letters $c$, $c'$ and $c''$ denote some constants whose meaning may change from line to line.
        
        \subsection{Some Useful Results}
        In our proof, we will constantly deal with the Hadamard product $u\circ v$ and the operation $(n^{-1}\,X^TXu)\circ v$ for two vectors $u,v\in\mathbb R^p$. 
        Therefore, we collect some useful properties in this section, some of which are consequences of the RIP condition. 
        
        The first property regarding the Hadamard product is a direct consequence of the H\"{o}lder inequality.
        \begin{lemma}\label{lem2}
        	For any two vectors $u$ and $v$ in $\mathbb{R}^{p}$,  we have:
        	\begin{equation}
        	\|u\circ v\| \leq\|u\|\|v\|_\infty.
        	\end{equation}
        \end{lemma}
        \begin{proof}
        	This follows since $\|u\circ v\|^2 = \sum_{j}u_j^2v_j^2 \leq \|v\|_\infty^2 \sum_{j}u_j^2=\|u\|^2\|v\|_\infty^2$.
        \end{proof}
        The second lemma shows that under the RIP, the product $(n^{-1}\,X^TXu)\circ v$ can be well-approximated by $u\circ v$ for all sparse vectors $u\in\mathbb R^p$ and any vector $v\in\mathbb R^p$. 
        \begin{lemma}\label{lem3}
        	Let $X$ be a matrix in $\mathbb{R}^{n\times p}$ that satisfies $(s+1,\delta)$-restricted isometry property (see Definition 2.1 in the paper). Then for any $s$-sparse vectors $u$ and any $v$ in $\mathbb{R}^{p} $, we have:
        	\begin{equation}
        	\|(n^{-1}\,X^TXu)\circ v-u\circ v\|_{\infty} \leq \delta\|u\|_2\|v\|_{\infty}.
        	\end{equation}
        \end{lemma}
        \begin{proof}
        	Let $D(v)$ be the diagonal matrix in $\mathbb{R}^{n\times p}$ with diagonal elements the same as components of $v$ correspondingly. Then $\|(n^{-1}\,X^TXu)\circ v-u\circ v\|_{\infty}$ can be represented as:
        	\begin{align*}
        	\max_{i=1,2,...,p} |e_i^T D(v) n^{-1}X^TXu-e_i^T D(v) u|,
        	\end{align*}
        	where $e_i$ is a $p$-dimensional vector whose $i$-th component is $1$ and $0$ elsewhere. Using the fact that $X$ satisfies $(s+1,\delta)$-RIP and $e_i^T D(v)$ is $1$-sparse, we have (see the remark right after Definition 2.1 in the paper):
        	\begin{align*}
        	\|(n^{-1}\,X^TXu)\circ v-u\circ v\|_{\infty}\leq  \max_{i=1,2,...,p} \delta\|e_i^T D(v)\|\|u\|= \delta\|u\|_2\|v\|_{\infty}.
        	\end{align*}
        \end{proof}
        Our third lemma considers the case when $u$ and $v$ are both arbitrary. 
        \begin{lemma}\label{lem4}
        	Let $X$ be a matrix in $\mathbb{R}^{n\times p}$ that satisfies $(2,\delta)$-restricted isometry property. Then for any vectors $u,\, v\in\mathbb R^p$, we have: 
        	\begin{equation}
        	\|(n^{-1}\,X^TXu)\circ v-u\circ v\|_{\infty} \leq \delta\|u\|_1\|v\|_{\infty}.
        	\end{equation}
        \end{lemma}
        \begin{proof}
        	Since we can decompose $u=\sum_j I_j u$, we have
        	\begin{align*}
        	\|(n^{-1}\,X^TXu)\circ v-u\circ v\|_{\infty}&=\max_{i=1,2,...,p} |e_i^T D(v)\, n^{-1}X^TXu-e_i^T D(v) u|\\
        	&\leq \sum_{j=1}^{p} \max_{i=1,2,...,p} |e_i^T D(v)\, n^{-1}X^TX \, I_j u-e_i^T D(v)\, I_j u| \\
        	&\leq \sum_{j=1}^{p} \max_{i=1,2,...,p} \delta \|e_i^T D(v)\| \|u_j\| \\
        	&\leq \delta\|u\|_1\|v\|_{\infty}.
        	\end{align*}
        \end{proof}
        Our fourth and fifth lemma consider the concentration behaviors about the noise terms.
        \begin{lemma}\label{lem6}
        	Let $w$ is a vector of independent sub-Gaussian random variables with scales bounded by $\sigma$, and $X \in \mathbb{R}^{n \times s}$, with $s<n$, then there exists a constant $c_0>0$, with probability $1-e^{-Ms}$ for any $M\geq 1$, such that for any $c\geq \sqrt{c_0}$:
        	\begin{equation}
        	\frac{1}{n}\|X^\top w\| \leq c \sigma  \sqrt{\frac{Ms\,\lambda_{\max}(\frac{X^\top X}{n})}{n}}.
        	\end{equation}
        \end{lemma}
        \begin{proof}
        	By directly denoting $c_0$ as the constant in Hanson-Wright inequality \citep{rudelson2013hanson}:
        	\begin{equation}
        	\begin{aligned}
        	\mathbb{P}(     w^\top X X^\top w \geq c^2 \sigma^2 n \,Ms \,\lambda_{\max}(\frac{X^\top X}{n})) 
        	\leq \exp \left(-\frac{c^2\sigma^2 n \,Ms\, \lambda_{\max}(\frac{X^\top X}{n})}{c_0 \sigma^2 \lambda_{\max}(X^\top X)}\right) \leq  e^{-Ms}.
        	\end{aligned}
        	\end{equation}
        \end{proof}

        \begin{lemma}\label{lem5}
        	Let $w$ is a vector of independent sub-Gaussian random variables with scales bounded by $\sigma$, and all $\ell_2$ norm of column vectors of $X \in \mathbb{R}^{n \times p}$ are normalized to $\sqrt{n}$, then with probability $1-2 p^{-1}$ such that:
        	\begin{equation}
        	\frac{1}{n}\|X^\top w\|_{\infty} \leq  \sigma \sqrt{\frac{4 \log p}{n}}.
        	\end{equation}
        \end{lemma}
        \begin{proof}
        	$\frac{1}{n}\|X^\top w\|_{\infty}$ is the maximum over $p$ sub-Gaussian random variables with scales bounded by $\sigma/\sqrt{n}$, so considering the union bound:
        	\begin{align*}
        	\mathbb{P}(\frac{1}{n}\|X^\top w\|_{\infty} \geq t) \leq 2 e^{\frac{-nt^2}{2\sigma^2}+\log p},
        	\end{align*}
        	the result is obtained by taking $t=\sigma \sqrt{\frac{4 \log p}{n}}$.
        \end{proof}

        \section{Proofs of Lemmas, Propositions and Theorems in the Paper}
        We prove the results in the paper one by one. We assume the largest signal strength is at the most constant level for simplicity, which implies $\kappa m \lesssim 1$. Note that the condition $\kappa m \lesssim 1$ can be satisfied under our two-step post debiasing process.

        \subsection{Proof of Lemma~\ref{Lem:global_min}}
        Let $\beta = g\circ l$. Since the least square problem $\min_{\beta\in\mathbb R^p} (2n)^{-1}\|X\beta -y\|^2$ is a convex optimization, it does not have local maximum and any local minimum is a global minimum. 
        
        Now suppose $f(g,l)$ has a local maximum $(\tilde g, \tilde l)$, meaning that there exists an open set $\mathcal B\in\mathbb R^{2p}$ centering at $(\tilde g, \tilde l)$ such for any $(g,l)\in \mathcal B$, $f(g,l) \leq f(\tilde g, \tilde l)$. It is easy to verify that the set $\mathcal B_\beta = \{g\circ l:\, (g,l)\in\mathcal B\}$ is also an open set in $\mathbb R^p$ centering at $\tilde \beta = \tilde g\circ \tilde l$. Consequently, $\tilde\beta$ is also a local maximum of $(2n)^{-1}\|X\beta -y\|^2$, which is a contradiction.
        
        Due to the same argument, any local minimum $(\bar g,\bar l)$ of $f(g,l)$ corresponds to a local minimum $\bar\beta =\bar g\circ \bar l$ of $(2n)^{-1}\|X\beta -y\|^2$. Since all local minimums of $(2n)^{-1}\|X\beta -y\|^2$ are global and satisfies $X^T\big(X\bar\beta-y\big) = 0$, $(\bar g,\bar l)$ must be a global minimum of $f(g,l)$ and satisfies 
        $X^T\big(X(\bar g\circ \bar l\,)-y\big) = 0$.
        
        Now let us prove the last part concerning saddle points. According to the previous paragraph and the discussion after lemma \ref{Lem:global_min}, a point $(g^\dagger,l^\dagger)$ is a saddle point (that is, neither local minimum nor local maximum) if and only if there exists some $j_0\in\{1,\ldots,p\}$ such that
        \begin{align*}
        g^\dagger_{j_0} = l^\dagger_{j_0}=0\qquad\mbox{and}\qquad R^\dagger_{j_0}\neq 0,
        \end{align*}
        where $R^\dagger = X^T\big(X(g^\dagger\circ l^\dagger) - y\big)\in\mathbb R^p$ (since otherwise this saddle point would satisfy $X^T\big(X(\bar g\circ \bar l\,)-y\big) = 0$, which is the sufficient and necessary condition for a global minimum).
        For any vector $u\in\mathbb R^p$, we use Diag$(u)$ to denote the $p$-by-$p$ diagonal matrix whose diagonals are components of $u$.
        By direct calculations, we can express the $2p$-by-$2p$ Hessian matrix $\nabla^2_{g,l} f(g^\dagger,l^\dagger)$ of $f(g,l)$ at $(g^\dagger,l^\dagger)$ as
        \begin{align*}
        \begin{pmatrix}
        \mbox{Diag}(l^\dagger) & 0\\
        0 & \mbox{Diag}(g^\dagger)
        \end{pmatrix}
        \begin{pmatrix}
        n^{-1}X^TX & n^{-1}X^TX\\
        n^{-1}X^TX & n^{-1}X^TX
        \end{pmatrix}
        \begin{pmatrix}
        \mbox{Diag}(l^\dagger) & 0\\
        0 & \mbox{Diag}(g^\dagger)
        \end{pmatrix}
        \\ +\ \
        \begin{pmatrix}
        0 & \mbox{Diag}(R^\dagger)\\
        \mbox{Diag}(R^\dagger) & 0
        \end{pmatrix}.
        \end{align*}
        Consequently, it is easy to verify that 
        \begin{align*}
        \big(e_{j_0}^T, \  -\mbox{sgn}(R^\dagger_{j_0}) \,e_{j_0}^T\big) \,\nabla^2_{g,l} f(g^\dagger,l^\dagger) 
        \begin{pmatrix}
        e_{j_0}\\ -\mbox{sgn}(R^\dagger_{j_0})\, e_{j_0}
        \end{pmatrix} =-2 (R^\dagger_{j_0})^2 <0,
        \end{align*}
        where $e_{j}$ denotes the vector whose $j$-th component is one and zero elsewhere for $j=1,\ldots,p$, and sgn$(\cdot)$ is the sign function. Consequently, there exists some nonzero vector $u$ such that the evaluation of the quadratic form $u^T\nabla^2_{g,l} f(g^\dagger,l^\dagger)u <0$, implying $\lambda_{\min}\big(\nabla^2_{g,l} f(g^\dagger,l^\dagger)\big)<0$.

        \subsection{Proof of Lemma~\ref{Lem:GD_global_min}}
        According to lemma \ref{Lem:global_min}, all saddle points of $f(g,l)$ are strict, and all local minimums are global. 
        Therefore, we can apply Theorem 4 in \cite{lee2016gradient} to finish the proof, which states that for any twice continuously differentiable function, if it has only strict saddle points,
        then gradient descent with a random initialization and sufficiently small constant step size almost surely converges to a local minimizer.

        \subsection{Gradient descent error dynamics}\label{gded}
        In this subsection, we provide two propositions for characterizing the dynamics of the gradient descent algorithm for minimizing the Hadamard product parametrized quadratic loss $f(g,l)$ defined in~\eqref{eq_opt}. Theorem~\ref{thm2} in the main content is a direct consequence of these two propositions. Due to the lack of knowledge on signs of nonzero components in signal $\beta^\ast$, the proof of Theorem~\ref{thm2} has the extra component in showing that $\beta_t$ tends to shoot towards the right direction (positive versus negative) during the first ``burn-in" stage before entering the second geometric convergence region. To deal with this extra complication, we introduce another reparametrization:
        $${a}_{t}=(g_t+l_t)/2\in\mathbb R^p,\quad\mbox{and}\quad
        {b}_{t}=(g_t-l_t)/2\in\mathbb R^p \quad\mbox{for}\quad  t=0,1\ldots.$$
        Notice that $\beta_t=a_t^2-b_t^2$, implying that the sign of each component $\beta_{t,j}$ is determined by the relative magnitude of $|a_{t,j}|$ and $|b_{t,j}|$. These new variables $(a_t,b_t)$ are especially convenient for understanding the dynamics of the sign of $\beta_t$ due to a more tractable updating formula:
        \begin{align*}
        a_{t+1}=a_{t}-\eta \,a_{t}\circ \big[n^{-1}X^T(X \beta_t-y)\big],\ \mbox{and} \
        b_{t+1}=b_{t}+\eta\, b_{t}\circ \big[n^{-1}X^T(X \beta_t-y)\big].
        \end{align*} 
        More precisely, in the first ``burn-in'' stage, we show that each component of the strong signal part $\beta_{t,S_1}$ increases at an exponential rate in $t$ until hitting $m/2$, while the weak signal and error part $\beta_{t,{S_1}^c}$ remains bounded
        by $\mathcal O(p^{-2})$. 
        
        \begin{proposition}[Stage one dynamics]\label{pro3.1}
        	Under the assumptions of Theorem~\ref{thm2},  with probability at least $1-\exp{(- c_1 \log p)}-\exp{(- c_2\,Ms)}$, there is a  positive constant $c_8$, such that for each $t< T_1=c_8\,\log({1/\alpha})/(\eta m)$, we have:
        	\begin{align*}
        	&\mbox{signal dynamics:} \qquad \mbox{sign}(\beta^\ast_{S_1})\circ \beta_{t,S_1} \geq \min\Big\{\frac{m}{2},\, \Big(1+\eta m/4\Big)^t\,\alpha^2 -\,\alpha^2\Big\}\, \mathbf 1\in\mathbb R^{s_1}, \\
        	&\qquad \max\big\{\|a_{t,S_1}\|_\infty,\, \|b_{t,S_1}\|_\infty\big\}\lesssim  1,\quad\mbox{and}\quad \|\beta_{t,S_1}-\beta^\ast_{S_1}\|\lesssim  \sqrt{s_1};\\[0.3em]
        	&\mbox{weak signals and error dynamics:} \,\, \|a_{t+1,S_1^c}\|_{\infty}\leq \big(1+ \mathcal{O}( 1/T_1)\big)\|a_{t,S_1^c}\|_{\infty} \lesssim 1/p,\\
        	&\quad\mbox{and} \qquad \|b_{t+1,S_1^c}\|_{\infty}\leq \big(1+ \mathcal{O}(1/ T_1) \big)\|b_{t,S_1^c}\|_{\infty} \lesssim 1/p.
        	\end{align*}
        \end{proposition}
        
        Let $\Theta_{LR}^G=\{(a,\,b)\in\mathbb R^{2p}:\, \mbox{sign}(\beta^\ast_{S_1})\circ (a_{S_1}^2- b_{S_1}^2) \geq (m/2) \, \mathbf 1,\,\max\big\{\|a_{S_1}\|_\infty,\, \\  \|b_{S_1}\|_\infty\big\}\lesssim  1,\, \|a_{S_1}^2- b_{S_1}^2-\beta^\ast_{S_1}\|\lesssim  \sqrt{s_1},\, \max\big\{\|a_{S_1^c}\|_\infty,\,\|b_{S_1^c}\|_\infty\big\} \lesssim 1/p\}$ as a good parameter region for $(a_t,b_t)$. Proposition~\ref{pro3.1} tells that for all $t<T_1$, iterate $(a_t,b_t)$ satisfies all constraints of $\Theta_{LR}^G$ except $\mbox{sign}(\beta^\ast_{S_1})\circ (a_{S_1}^2- b_{S_1}^2) \geq m/2$, but will enter this good region in at most $\mathcal O(\log({m \alpha^{-2}})/(\eta m))$ iterations. 
        
        The second stage starts when $(a_t,b_t)$ first enters $\Theta_{LR}^G$. The next result summarizes the behavior of the gradient dynamics in the second stage---once it enters $\Theta_{LR}^G$, it will stay in $\Theta_{LR}^G$ for a long time, where $\beta_t=a_t^2-b_t^2$ converges toward $\beta^\ast$ at a linear rate up to the statistical error $\varepsilon_n$ and then stay in an $\mathcal O(\varepsilon_n)$ neighborhood of $\beta^\ast$ up to time $\Theta\big(1/(\eta \sigma \sqrt{\log p/n} )\big)$. 
        \begin{proposition}[Stage two dynamics]\label{pro3.4}
        	Under the assumptions of Proposition~\ref{pro3.1}, if $(a_t,b_t) \in \Theta_{LR}^G$ with $t \leq \Theta (\max\{\log({1/\alpha})/(\eta m), 1/(\eta \sigma \sqrt{\log p/n} )\})$, then $(a_{t+1},b_{t+1}) \in \Theta_{LR}^G$, and $\beta_{t}=a_t^2-b_t^2$ satisfies
        	\begin{equation*}
        	\|\beta_{t+1}-\beta^\ast\|^2\leq (1-\eta m)\,\|\beta_t-\beta^\ast\|^2+ \varepsilon_n.
        	\end{equation*} 
        \end{proposition}
        A combination of these two propositions leads to a proof of Theorem~\ref{thm2}. Therefore, the time interval $\big[c_3  \log(1/\alpha)/(\eta  m) ,\,  c_4/(\max\{  \log(1/\alpha)/(\eta  m), 1/( \eta \sigma\sqrt{\log p/n})\} \big]$ would be the theoretical ``best solution region'' corresponding to the stabilized region in Figure~\ref{fig:31}.
        
        \subsection{Proof of Proposition \ref{pro3.1}}
        Throughout the proof, we let $\epsilon=\max\{  \alpha^2, \sigma^2 {  Ms_1/n}, \sigma^2 {s_2\log p/n}\}$ and \\ $\tau=\max \{\delta \alpha, \sigma \sqrt{{\log p/n}}\}$ for any $M\geq 1$.
        Recall we have the updating formula:
        \begin{eqnarray*}
        	a_{t+1}=a_{t}-\eta \,a_{t}\circ \big[n^{-1}X^T(X \beta_t-y)\big],\ \mbox{and} \
        	b_{t+1}=b_{t}+\eta\, b_{t}\circ \big[n^{-1}X^T(X \beta_t-y)\big],
        \end{eqnarray*}
        with $\beta_t=a_t^2-b_t^2$. We use induction to show that for each $t\leq T_1$,
        \begin{align}
        \|a_{t,S_1^c}\|_{\infty}\lesssim 1/p,\ &&\|b_{t,S_1^c}\|_{\infty}\lesssim 1/p \label{eq3.1.1},\\
        \|a_{t,S_1}\|_{\infty}\lesssim 1, &&\|b_{t,S_1}\|_{\infty}\lesssim 1 \label{eq3.1.2},\\
        \|\beta_{t,S_1}-\beta^\ast_{S_1}\|\lesssim \sqrt{s} \label{eq3.1.3},
        \end{align}
        where the set $S_1^c$ is the union of weak signals and errors. When $t=0$, we have $g_0=\alpha \mathbf{1}$, $l_0=0$. Therefore, under the assumption $\alpha\lesssim 1/p$, we have $\|a_{0,S_1^c}\|_{\infty}\lesssim 1/p $, $\|a_{0,S_1}\|_{\infty}\lesssim 1$, and similar bounds for $b$. Now suppose for time $t<T_1$, we have \eqref{eq3.1.1}-\eqref{eq3.1.3}. We divide into 3 steps to show the same bounds hold for time $t+1$. Note that the following analysis is still based on the conditional events: 
        \begin{equation}
        \frac{1}{n}\|X_{S_1}^\top w\| \leq c \sigma \sqrt{\frac{Ms_1}{n}} \quad \mbox{and} \quad \frac{1}{n}\|X^\top w\|_{\infty} \leq  \sigma \sqrt{\frac{4 \log p}{n}} ,
        \end{equation}  
        where the first concentration is based on lemma~\ref{lem6} and the RIP of $X^\top X/n$. Also recall the following conditions on the step size $\eta$, and RIP constant $\delta$: 
        $\eta\lesssim 1/(\kappa \log(1/\alpha))$, and $\delta \lesssim 1/(\kappa s^{1/2} \log (1/\alpha))$.
        
        \begin{proposition}[Step 1: Analyzing Error Dynamics:]\label{pro3.12}
        	Based on RIP and the assumption for constants in theorem~\ref{thm2}, assume the induction hypothesis~\eqref{eq3.1.1}-\eqref{eq3.1.3} hold for time $t$ with $t<T_1$, then we have:
        	\begin{align*}
        	\|a_{t+1,S_1^c}\|_{\infty}&\leq (1+c\eta m/\log(p))\|a_{t+1,S_1^c}\|_{\infty},\\
        	\|b_{t+1,S_1^c}\|_{\infty}&\leq (1+c\eta m/\log(p))\|b_{t+1,S_1^c}\|_{\infty}.
        	\end{align*}
        	When $\|\beta_{t,S_1}-\beta_{S_1}^\ast\| \leq c \sqrt{\epsilon}$, let $\tau =\max \{\sqrt{{\log p/n}},  \delta \alpha\}$, then
        	\begin{align*}
        	\|a_{t+1,S_1^c}\|_{\infty}&\leq (1+c\eta\tau)\|a_{t+1,S_1^c}\|_{\infty},\\
        	\|b_{t+1,S_1^c}\|_{\infty}&\leq (1+c\eta\tau)\|b_{t+1,S_1^c}\|_{\infty}.
        	\end{align*}
        \end{proposition}

        \begin{proof}
        	By the definition of $a_t$ and $b_t$, we have the updating rules:
        	\begin{align*}
        	a_{t+1,S_1^c}=a_{t,S_1^c}-\eta a_{t,S_1^c}\circ n^{-1}X_{S_1^c}^T(X \beta_t-y),\\
        	b_{t+1,S_1^c}=b_{t,S_1^c}+\eta b_{t,S_1^c}\circ n^{-1}X_{S_1^c}^T(X \beta_t-y).
        	\end{align*} 
        	Denote $r_t=\beta_t-\beta^*$. Note that $S_1^c = S^c \cup S_2 $. For $\|a_{t+1,S^c}\|_{\infty}$, we bound the term $n^{-1}X_{S^c}^T(X \beta_t-y) \circ a_{t,S^c}$ in the updating formula:
        	\begin{align*}
        	&\|n^{-1}X_{S^c}^T(X \beta_t-y) \circ a_{t,S^c}\|_{\infty}= \|I_{S^c} a_{t} \circ \left(n^{-1}X^T\left(X(\beta_t-\beta^\ast)-w\right)\right)\|_\infty \\
        	&\stackrel{(i)}{=} \|I_{S^c} a_{t}\circ \left( (n^{-1}X^TX(I_{S} r_t)-(I_{S} r_t)+n^{-1}X^TX(I_{S^c} r_t)-n^{-1}X^Tw \right)\|_\infty \\
        	&\stackrel{(ii)}{\leq} \|a_{t,S^c}\|_{\infty} \left(\delta\|I_{S} r_t\|+\|(I_{S^c} r_t)\|_{\infty}+\delta \|(I_{S^c} r_t)\|_1+ c \sigma \sqrt { \frac { \log  p  } { n } } \right) \\
        	&\stackrel{(iii)}{\leq}\|a_{t,S^c}\|_{\infty}(c \frac{m}{\log (p)}),
        	\end{align*} 
        	where in step~$(i)$ we use the fact that $I_{S^c} a_{t}  \circ I_{S} r_t=0$; step~$(ii)$ follows by lemma~\ref{lem3}:
        	\begin{align*}
        	\|I_{S^c} a_{t} \circ  \left((n^{-1}X^TX(I_{S_1} r_t))-(I_{S_1} r_t)\right)\|_{\infty} \leq \|a_{t,S^c}\|_{\infty} \left(\delta\|I_{S_1} r_t\|  \right) \lesssim \delta\sqrt{s} \kappa m \|a_{t,S^c}\|_{\infty},
        	\end{align*}
        	lemma~\ref{lem4}:       
        	\begin{align*}
        	\|I_{S^c} a_{t} \circ  \left(n^{-1}X^TX(I_{S^c} r_t)-I_{S^c} r_t+I_{S^c} r_t \right)\|_{\infty} &\leq \|a_{t,S^c}\|_{\infty} \left(\delta\|I_{S^c} r_t\|_1+\|(I_{S^c} r_t)\|_{\infty}\right)\\
        	& \lesssim \frac{\delta}{p} \|a_{t,S^c}\|_{\infty},
        	\end{align*}
        	and lemma~\ref{lem5}:
        	\begin{align*}
        	\|I_{S^c} a_{t} \circ  \left(n^{-1}X^Tw \right)\|_{\infty} \leq \|a_{t,S^c}\|_{\infty} c\sigma \sqrt { \frac { \log  p  } { n } }.
        	\end{align*}
        	In step~$(iii)$, we applied the induction hypothesis~\eqref{eq3.1.1}. Finally, note that $\delta\sqrt{s} \kappa m \lesssim {m/\log(p)}$ and $\sigma \sqrt {  { \log  p /n } } \lesssim   m/\log(p)$ by the minimal strength of $m$, then an application of the triangle inequality leads to the claimed bound on $\|a_{t+1,S^c}\|_{\infty}$ and then similarly $\|b_{t+1,S^c}\|_{\infty}$.\par 
        	
        	Similarly, for $\|a_{t+1,S_2}\|_\infty$, 
        	\begin{align*}
        	&\|n^{-1}X^T(X\beta_t-y) \circ I_{S_2} a_{t}\|_{\infty}= \|I_{S_2} a_{t} \circ \left(n^{-1}X^T\left(X r_t-w\right)\right)\|_\infty \\
        	&\stackrel{(i)}{=}\|I_{S_2} a_{t} \circ ((n^{-1}X^TX(I_{S} r_{t})-(I_{S} r_{t})+(I_{S_2} r_{t})\\ &+n^{-1}X^TX(I_{S^c} r_{t})-(I_{S^c} r_{t})-n^{-1}X^T w )\|_\infty \\
        	&\stackrel{(ii)}{\leq} \|a_{t+1,S_2}\|_{\infty} \left(\delta\|I_{S} r_{t}\|+\|I_{S_2} r_{t}\|_{\infty}+\delta \|I_{S^c} r_{t}\|_1+ c\sigma \sqrt { \frac { \log  p  } { n } } \right) \\
        	&{\leq}\|a_{t+1,S_2}\|_{\infty}(c \, m/\log(p)),
        	\end{align*} 
        	where step~$(i)$ follows from the orthogonality between $S_2$ and $S_1, S^c$. The last inequality follows from the induction hypothesis and the definition of weak signals. \par
        	
        	Then, when $\|\beta_{t,S_1}-\beta_{S_1}^\ast\|\leq c\sqrt{\epsilon}$:
        	\begin{align*}
        	\|n^{-1}X^T(X\beta^2-y) \circ I_{S^c} a_{t}\|_{\infty}&\leq\|a_{t,S^c}\|_{\infty} (\delta\|I_{S} r_{t}\|+\|I_{S^c} r_{t}\|_{\infty}+\delta \|I_{S_c} r_{t}\|_1+ c\sigma \sqrt { \frac { \log  p  } { n } }) \\
        	&\leq c \tau \|a_{t,S^c}\|_{\infty},
        	\end{align*} 
        	where the last inequality follows from $\delta \cdot \max \{ \alpha, \sigma \sqrt{{s \log p/ n}}\} \leq \max \{ \delta  \alpha, \sigma \sqrt{{\log p/n}} \}$ and $\delta p \alpha^2 \lesssim \delta \alpha$.
        	\begin{align*}
        	\|n^{-1}X^T(X\beta_t-y) \circ I_{S_2} a_{t}\|_{\infty}& {\leq} \|a_{t+1,S_2}\|_{\infty} \left(\delta\|I_{S} r_{t}\|+\|I_{S_2} r_{t}\|_{\infty}+\delta \|I_{S^c} r_{t}\|_1+ c\sigma \sqrt { \frac { \log  p  } { n } } \right)\\
        	&\leq c \tau \|a_{t+1,S_2}\|_{\infty},
        	\end{align*} 
        	which is based on $\|\beta_{t,S_2}-\beta^{\ast}_{S_2}\|_{\infty} \leq c \max\{\alpha^2,\sigma \sqrt {  { \log  p /n } }\} $, since $\|\beta_{t,S_2}\|_{\infty} \leq c \alpha^2$ and the signals in set $S_2$ are weak. 
        	A similar proof can be applied to the dynamics of $b_{t}$.
        \end{proof}

        \mbox{Step 2: Analyzing Signal Dynamics:}
        The proof becomes more complicated for the signal dynamics. Our goal is to show:
        \begin{itemize}
        	\item  if the true signal $\beta_i^\ast$ is positive and $\beta_{t,i}<\beta_i^*/2$, then $a_{t,i}^2$ increases and $b_{t,i}^2$ decreases, both in an exponential rate. Overall, the sign of $i$th component $\beta_{t,i}=a_{t,i}^2-b_{t,i}^2$ of $\beta_t$ tends to grow to positive in an exponential rate;
        	\item if the true signal $\beta_i^\ast$ is negative and $\beta_{t,i}>\beta_i^*/2$, then $a_{t,i}^2$ decreases and $b_{t,i}^2$ increases, both in an exponential rate.  Overall, the sign of $i$th component $\beta_{t,i}=a_{t,i}^2-b_{t,i}^2$ of $\beta_t$ tends to fall to negative in an exponential rate.
        \end{itemize}
        Recall that we assumed that the RIP constant $\delta\lesssim 1/(\kappa \sqrt{s}\log ({1/\alpha}))$, step size $\eta\lesssim 1/(\kappa \log({1/\alpha}))$, and $T_1=\Theta({ \log(1/\alpha)/(\eta m}))$.
        \begin{proposition}\label{pro3.2}
        	Under assumptions in theorem~\ref{thm2} and the induction hypothesis~\eqref{eq3.1.1}-\eqref{eq3.1.3} at $t<T_1$,  when $|\beta_{t,i}-\beta^\ast_{t,i}|\geq \frac{1}{2}|\beta_i^*|$, then for any $i\in S$:
        	\begin{align*}
        	\beta_{t,i}\geq (1+c\,\eta\,\beta^\ast_{i})^t \alpha^2-c'\alpha^2,\ \ \text{if} \ \beta^\ast_{i}>0;\\
        	\beta_{t,i}\leq (1-c\,\eta\,\beta^\ast_{i})^t (-\alpha^2)+c'\alpha^2,\ \ \text{if}\  \beta^\ast_{i}<0.
        	\end{align*}
        \end{proposition}
        \begin{proof}
        	First we approximate $(n^{-1}X^TXu)\circ v$ by $u\circ v$ based on the RIP condition via lemmas~\ref{lem3} and \ref{lem4},
        	\begin{align}
        	&\|n^{-1}X^T(X\beta_t-y)-(I_{S_1}\beta_{t}-I_{S_1}\beta^\ast)\|_{\infty} \nonumber\\
        	\leq&\, \|n^{-1}X^TX (I_{S_1^c}r_{t})\|_{\infty}+\delta\|r_{t,S_1}\|  +\|X^Tw\|_\infty 
        	\lesssim \frac{m}{\log (p)},    \label{eq4.2.1}
        	\end{align}
        	implying that under the condition $ m\lesssim 1$, we have
        	\begin{align*}
        	&\|n^{-1}X^TX(\beta_t-y)\|_{\infty}\\
        	\leq&\, \|\beta_{t,S_1}-\beta^\ast_{S_1}\|_{\infty} +
        	\|n^{-1}X^T(X (I_{S_1^c}r_{t})-w)\|_{\infty}+\delta\|\beta_{t,S_1}-\beta^\ast_{S_1}\| 
        	\lesssim 1,    
        	\end{align*}
        	where the last inequality uses $\|\beta_{t,S_1^c}\|_{\infty}\lesssim 1/p$ and $\|\beta_{t,S_1}-\beta^\ast_{S_1}\|\lesssim \sqrt{s}$.
        	
        	In order to analyze $\beta_{t,S_1}=a_{t,S_1}^2-b_{t,S_1}^2$, let us focus on $a_{t,S_1}^2$ and $b_{t,S_1}^2$, separately. According to the updating rule of $a_{t,S_1}$, we have
        	\begin{align*}
        	&a_{t+1,S_1}^2=a_{t,S_1}^2-2\eta a_{t,S_1}^2 \circ [n^{-1}X^T(X\beta_{t}-y)]_{S_1}+\eta ^2 a_{t,S_1}^2 \circ [n^{-1}X^T(X\beta_{t}-y)]_{S_1}^2,
        	\end{align*}
        	where recall that for a vector $a\in\mathbb R^p$, $a_{S_1}$ denote the sub-vector of $a$ with indices in $S_1$. Applying lemmas~\ref{lem3} with $v=\mathbf 1$, we obtain
        	\begin{align*}
        	&\|a_{t+1,S_1}^2-a_{t,S_1}^2-2\eta a_{t,S_1}^2(\beta_{t,S_1}-\beta^\ast_{S_1})\|_{\infty}\lesssim \eta\frac{m}{\log (p)}+\eta^2 \kappa^2 m^2\stackrel{(i)}{\lesssim} \eta\frac{m}{\log (p)}.
        	\end{align*}
        	where in step $(i)$ we used $\eta \kappa m \lesssim {m/\log (p)}$ sand $\kappa m \lesssim 1$. Since $\eta  m \leq 1/2$, $ {m/\log (p)} \leq 1/2$, we have $a_{t,i}^2/a_{t+1,i}^2 \leq 4$ for $i\in S_1$. Therefore, we can obtain an element-wise bound for $\xi_t=(\xi_{t,i})_{i\in S_1}$,
        	\begin{equation*}
        	\xi_{t,i}:\,=1-a_{t,i}^2/a_{t+1,i}^2 \circ(2\eta ({\beta}_{t,i}-\beta^\ast_i)), 
        	\end{equation*} 
        	as $\|\xi_t\|_{\infty}\lesssim \eta{m/\log (p)}$.
        	Equivalently, we can write
        	\begin{align}
        	a_{t+1,i}^2= a_{t,i}^2 (1-2\eta ({\beta}_{t,i}-\beta^\ast_i))+\xi_{t,i} a_{t+1,i}^2. \label{eq4.2.2}
        	\end{align} 
        	Now let us divide into two cases depending on the sign of $\beta^\ast_i,\ i\in S_1$:
        	
        	\emph{Case $\beta_i^* >0$:} When ${\beta}_{t,i}-\beta^\ast_i\leq -\beta^\ast_i/2$, since $\beta^\ast_i \geq m$, we have by equation~\eqref{eq4.2.2},
        	\begin{align*}
        	a_{t+1,i}^2 &\geq \frac{a_{t,i}^2(1+\eta \beta^\ast_i)}{1+c\eta\frac{m}{\log (p)}} \geq  a_{t,i}^2(1+\eta \beta^\ast_i)(1-c\eta\frac{\beta_i^*}{\log (p)} )\geq  a_{t,i}^2(1+\eta \beta^\ast_i/4),
        	\end{align*}
        	where the last inequality follows since $1/\log(p)\leq 1/2$ and $\eta \beta^\ast_i \leq 1/2$. Similarly, we can analyze $b_{t,S}^2$ to get
        	\begin{align*}
        	b_{t+1,i}^2 &\leq \frac{b_{t,i}^2(1-\eta \beta^\ast_i)}{1-c\eta\delta\sqrt{s}}\leq  b_{t,i}^2(1-\eta \beta^\ast_i)(1+c\eta\beta_i^*/\log(p))\leq  b_{t,i}^2(1-\eta \beta^\ast_i/4).
        	\end{align*}
        	Therefore, $a_{t+1,i}^2$ increases at an exponential rate faster than the noise term $a_{t+1,S_1^c}$ while $b_{t+1,i}^2$ decreases to zero at an exponential rate, and when $a_{t+1,i}$ increases to $\beta_{i}^\ast/2$, $b_{t+1,i}$ decreases to $O( \alpha^4)$ correspondingly. A combination of these two leads to the first claimed bound for $\beta_{i}^\ast>0$.
        	
        	\emph{Case $\beta_i^* <0$:} The analysis for the case is similar: when ${\beta}_{t,i}-\beta^\ast_i\geq -\beta^\ast_i/2$, we have:
        	\begin{align*}
        	a_{t+1,i}^2 \leq  a_{t,i}^2(1-\eta \beta^\ast_i/4),\quad\mbox{and}\quad b_{t+1,i}^2 \geq  b_{t,i}^2(1+\eta \beta^\ast_i/4),
        	\end{align*}
        	which leads to the second claimed bound for $\beta_{i}^\ast<0$.
        \end{proof}
        
        As a consequence of the proof in this step, after at most $T\geq \Theta ( \frac{\log(m/\alpha^2)}{\eta m})$ iterations, we are guaranteed to have have $|{\beta}_{T,i}|\geq |\beta^\ast_i|/2$ with $sign({\beta}_{T,i})=sign(\beta^\ast_i)$ and $\min\{a_{T,i}^2, b_{T,i}^2\} \leq c \alpha^4$.
        
        \mbox{Step 3: Prove Induction Hypothesis:}
        Our last piece shows that the induction hypothesis is kept for $t+1$. This proposition combined with propositions~\ref{pro3.12}-\ref{pro3.2} leads to a proof for proposition \ref{pro3.1} in the paper.
        \begin{proposition}\label{pro3.3}
        	Under assumptions in theorem~\ref{thm2},  the induction hypothesis~\eqref{eq3.1.1}-\eqref{eq3.1.3} hold at $t<T_1$, we have:
        	\begin{align*}
        	&\|a_{t+1,S_1^c}\|_{\infty}  \lesssim 1/p,\quad \|b_{t+1,S_1^c}\|_{\infty}\lesssim 1/p ,\\
        	&\|a_{t+1,S_1}\|_{\infty} \lesssim 1, \quad\|b_{t+1,S_1}\|_{\infty}\lesssim 1 ,\\
        	&\mbox{and}\quad\|{\beta}_{t+1,S_1}-\beta^\ast_{S_1}\|\lesssim 1.
        	\end{align*}
        	Then for $T_1<t<T_2$, the above induction still holds.
        \end{proposition}
        \begin{proof}
        	Our induction hypothesis 
        	\begin{align*}
        	&\|a_{t,S_1^c}\|_{\infty}\lesssim 1/p,\quad \|b_{t,S_1^c}\|_{\infty}\lesssim 1/p,\\
        	&\|a_{t,S_1}\|_{\infty}\lesssim 1, \quad\|b_{t,S_1}\|_{\infty}\lesssim 1,
        	\end{align*} 
        	implies $\|\beta_{t,S_1^c}\|_{\infty}\lesssim 1/p^2$ and $\|\beta_{t,S_1^c}-\beta^\ast_{t,S_1^c}\|\lesssim \sqrt{s}$. By updating rules:
        	\begin{align*}
        	a_{t+1}=a_{t}-\eta a_{t}\circ n^{-1}X^T(X {\beta}_t-y),\\
        	b_{t+1}=b_{t}+\eta b_{t}\circ n^{-1}X^T(X {\beta}_t-y),
        	\end{align*}
        	and our condition $\eta m/\log(p) \lesssim 1/T_1$.       For weak signals and error components $a_{t+1,S_1^c}$, we have:
        	\begin{align*}
        	\|a_{t+1,S_1^c}\|_{\infty}&\leq (1+c\,\eta m/\log(p))\|a_{t,S_1^c}\|_{\infty} 
        	\stackrel{(i)}{\leq} (1+c/T_1)^{T_1}\|a_0\|_{\infty}  \\
        	&\leq e^c \|a_0\|_{\infty} \lesssim \alpha \lesssim 1/p,
        	\end{align*}
        	where in step $(i)$ we apply $ {m/ \log (p)} \lesssim 1/(\eta T_1)$. When $\|\beta_{t,S_1}-\beta^\ast_{S_1}\| \leq \sqrt{\epsilon}$, since $T_2=c/(\eta \tau)$, when $T_1<t<T_2$:
        	\begin{align*}
        	\|a_{t+1,S_1^c}\|_{\infty}&\leq (1+c\,\eta \tau )\|a_{t+1,S_1^c}\|_{\infty} 
        	{\leq} (1+c/T_2)^{T_2} \cdot \alpha \\
        	&\leq e^c \cdot \alpha \lesssim \alpha \lesssim 1/p.
        	\end{align*}
        	For strong signal component $a_{t+1,S_1}$, $b_{t+1,S_1}$,  we have the component-wise bound $a_{t,i}^2+b_{t,i}^2=g_{t,i}^2+l_{t,i}^2\geq |g_{t,i} \circ l_{t,i}|=|\beta_{t,i}|$, then
        	\begin{align}\label{3.4.3}
        	a_{t,i}^2+b_{t,i}^2-|\beta_{t,i}|=2\min\{a_{t,i}^2,b_{t,i}^2\}\stackrel{(i)}{\lesssim} \alpha^4,
        	\end{align}
        	where $(i)$ follows by proof of proposition \ref{pro3.2}, since the smaller one of $\{a_{t,i}^2,b_{t,i}^2\}$ will be its initial value decreasing at an exponential rate until the rate $\alpha^4$. If $\beta_{t,i}>0$ such that $a_{t,i}$ increases to large numbers while $b_{t,i}$ converges to zero,  then based on $\beta_{t,i} \geq a_{t,i}^2-\alpha^4$, we can apply equation~\eqref{eq4.2.2} to get element-wise upper bound for $i \in S_1$:
        	\begin{align*}
        	a^2_{t+1,i}&\leq a^2_{t,i} (1-2\eta (a_{t,i}^2-\beta^\ast_i)) +c\,\eta\, \delta\,\sqrt{s}\, a_{t,i}^2.
        	\end{align*}
        	Since $\delta \sqrt{s}\lesssim 1$, and the function $f(x)=2\eta x^2-c\eta x$ is nonnegative when $x\gtrsim 1$. Therefore, as long as $\|a_{t,i}\|_\infty \gtrsim 1$, we always have
        	$\|a_{t+1,i}\|_{\infty}\leq \|a_{t,i}\|_{\infty} \lesssim 1$. Similarly, $\|b_{t,i}\|_{\infty} \lesssim 1$. Then we have
        	\begin{align*}
        	\|\beta_{t+1,S_1}-\beta^\ast_{S_1}\|_\infty &\leq  \|a_{t+1,S_1}^2\|_{\infty}+\|b_{t+1,S}^2\|_{\infty}+\|\beta^\ast_{S_1}\|_\infty \lesssim 1.
        	\end{align*}
        	Now we have:
        	\begin{align*}
        	&\|a_{t+1,S_1^c}\|_{\infty}\lesssim 1/p,\quad \|b_{t+1,S_1^c}\|_{\infty}\lesssim 1/p ,\\
        	&\|a_{t+1,S_1}\|_{\infty}\lesssim 1, \quad\|b_{t+1,S_1}\|_{\infty}\lesssim 1.
        	\end{align*}
        	These inequalities also imply $\|\beta_{t+1,S_1^c}\|_{\infty}\lesssim 1/p^2$ and $\|\beta_{t+1,S_1}-\beta^\ast_{t+1,S_1}\|\lesssim \sqrt{s}$. Then similarly for $T_1<t<T_2$, since the error updating still holds, equation~\eqref{eq4.2.2} can still help us to provide the same proof, so we also have:
        	\begin{align*}
        	\|a_{t+1,S_1}\|_{\infty}\lesssim 1, \quad\|b_{t+1,S_1}\|_{\infty}\lesssim 1.
        	\end{align*} 
        \end{proof}
        \subsection{Proof of Proposition~\ref{pro3.4}}
        \begin{proof}
        	First, we can decompose $\|\beta_{t+1}-\beta^\ast\|^2=\|a_{t+1}^2-b_{t+1}^2-\beta^\ast\|^2$ into 5 parts according to the order of $\eta$:
        	\begin{align*}
        	\|\beta_{t+1}-\beta^\ast\|^2&=\|(a_{t}-\eta a_{t} \circ n^{-1}X^T(X r_t-w))^2-(b_{t}+\eta b_{t} \circ n^{-1}X^T(X r_t-w))^2-\beta^\ast\|^2, \\
        	&=\|\beta_t-2\eta(a_t^2+b_t^2)\circ n^{-1}X^T(X r_t-w)+\eta^2\beta_t\circ (n^{-1}X^T(X r_t-w))^2-\beta^\ast\|^2 ,\\
        	&=\|\beta_{t}^2-\beta^\ast\|^2-4\eta\langle (a_t^2+b_t^2)\circ n^{-1}X^T(X r_t-w),\beta_t-\beta^\ast\rangle,\\
        	&\quad+4\eta^2 \| (a_t^2+b_t^2)\circ n^{-1}X^T(X r_t-w)\|^2\\
        	&+2\eta^2 \langle \beta_t\circ (n^{-1}X^T(X r_t-w))^2 ,\beta_{t}-\beta^\ast\rangle, \\
        	&\quad-4\eta^3 \langle \beta_t\circ (n^{-1}X^T(X r_t-w))^2 ,(a_t^2+b_t^2)\circ n^{-1}X^T(X r_t-w)\rangle,\\
        	&\quad+\eta^4 \|\beta_t\circ (n^{-1}X^T(X r_t-w))^2\|^2.
        	\end{align*} 
        	First note that we have equation~\eqref{3.4.3}, so we have the decomposition:
        	\begin{align}\label{decomp}
        	\|\beta_{t+1}-\beta^\ast\|^2&=\|\beta_{t+1,S_1}-\beta_{S_1}^\ast\|^2+\|\beta_{t+1,S_2}-\beta_{S_2}^\ast\|^2+\|\beta_{t+1,S^c}\|^2 \nonumber \\
        	&\leq \|\beta_{t+1,S_1}-\beta_{S_1}^\ast\|^2+ cs_2 \frac{\log p}{n}+ c' p \alpha^4,
        	\end{align}
        	which is based on the orthogonality between different sub vectors and the induction hypothesis, so we only need to calculate the errors induced by the strong signals.  \\
        	For the gradient, we have:
        	\begin{align}
        	\|n^{-1}X^T(Xr_t-w)\|_{\infty}\leq &\, \|\beta_{t,S}-\beta^\ast\|_{\infty}+\delta\|\beta_{t,S}-\beta^\ast\| \nonumber \\
        	&+\|\beta_{t,S^c}\|_{\infty}+\delta p \|\beta_{t,S^c}\|_{\infty} +\|X^Tw\|_\infty \leq c; \label{3.1.13}\\
        	\|n^{-1}X^T(X (I_{S_1^c} r_t) -w)\|_{\infty}\leq&\,  c \delta p \alpha^2+ c \sigma \sqrt{\frac{\log p}{n}}. \label{eq3.1.15}
        	\end{align}
        	Note that $I_{S_1} \beta_t$ is $s_1$ sparse, so we have the following core error term  characterization:
        	\begin{align}
        	\|I_{S_1} \beta_t \circ n^{-1}X^T(X r_{t,S_1^c}-w)\| \leq&\,  cs \delta^2 p^2 \alpha^4+ c' \sigma^2 M s_1 \frac{\log p}{n} \leq c (\alpha^2+\sigma^2 \frac{Ms_1}{n}). \label{eq3.1.16}
        	\end{align}
        	Similarly, our analysis for errors is based on equation~\eqref{eq3.1.15} and~\eqref{eq3.1.16}.
        	
        	For the order one term, recall the bound
        	\begin{align*}
        	&\langle (a_t^2+b_t^2)\circ n^{-1}X^T(X r_t-w),I_{S_1} r_{t}\rangle \\
        	& \stackrel{(i)}{\geq} \langle |I_{S_1} \beta_t|\circ n^{-1}X^T (X r_t-w), I_{S_1} r_{t}\rangle -s \alpha^4 \\
        	&=\langle |I_{S_1} \beta_t|\circ n^{-1}X^T X (I_{S_1} r_t), I_{S_1} r_{t}\rangle 
        	+\langle |I_{S_1} \beta_t|\circ n^{-1}X^T (X (I_{S_1^c} r_t)-w), I_{S_1}r_{t}\rangle\\
        	&\geq \left\|\,r_{t,S_1}\circ \sqrt{|\beta_{t,S_1}|}\,\right\|^2-\delta  \|\,r_{t,S_1}\|^2\\
        	&-\|\sqrt{|\beta_{t,S_1}|} \circ n^{-1}X^T (X (I_{S_1^c} r_t)-w)\|^2/2 -\| \sqrt{|\beta_{t,S_1}|} \circ \,r_{t,S_1}\|^2/2 \\
        	&\geq \frac{1}{2}\left\|\,r_{t,S_1}\circ \sqrt{|\beta_{t,S_1}|}\,\right\|^2-\delta  \|\,r_{t,S_1}\|^2 -c (\alpha^2+\sigma^2 \frac{Ms_1}{n}).       
        	\end{align*} 
        	where $(i)$ is based on Cauchy-Schwarz inequality and we use the fact $(a_t,b_t)\in\Theta_{LR}^G$, lemma~\ref{lem2}, equation~\eqref{eq3.1.16}.
        	
        	For the rest of terms, we can see they are all based on the bound: 
        	\begin{align*}\label{eq:3.1.17}
        	&\|I_{S_1} \sqrt{\beta_t}   \circ n^{-1}X^T(Xr_t-w)\|^2\\ 
        	&\leq \|I_{S_1} \sqrt{\beta_t}   \circ n^{-1}X^TX (I_S r_t)\|^2+\|I_{S_1} \sqrt{\beta_t} \circ n^{-1}X^T(I_{S_1^c} r_t-w)\|^2 \\
        	&\lesssim  c \|r_{t,S_1}\|^2+c' (\alpha^2+\sigma^2 \frac{Ms_1}{n}).
        	\end{align*} 
        	
        	First consider the order two terms $\langle \beta_t\circ (n^{-1}X^TX r_t)^2 ,I_{S_1}(\beta_{t}-\beta^\ast)\rangle$ and $\| I_{S_1}(a_t^2+b_t^2)\circ n^{-1}X^TX r_t\|^2$. For the former one, since $\|r_t\|_{\infty}\lesssim 1$, we have
        	\begin{align*}
        	\langle \beta_t\circ (n^{-1}X^TX r_t)^2 ,I_{S_1} r_{t}\rangle &= \langle I_{S_1}\beta_t\circ n^{-1}X^T(X r_t-w), n^{-1}X^T(X r_t-w) \circ r_t)\rangle\\
        	&\leq c \|I_{S_1}\sqrt{\beta_t} \circ n^{-1}X^T(X r_t-w)\|^2 \leq  c \|r_{t,S_1}\|^2+c' (\alpha^2+\sigma^2 \frac{Ms_1}{n}).
        	\end{align*}
        	where we use the condition $(a_t,b_t)\in\Theta_{LR}^G$. 
        	The analysis of the latter one $\| I_{S_1}(a_t^2+b_t^2)\circ n^{-1}X^TX r_t\|^2$  is similar based on $\|a_t^2+b_t^2 - \beta_t\|_{\infty}\lesssim \alpha^2$ and $\|\sqrt{\beta_t} \| \lesssim 1 $: 
        	\begin{align*}
        	&\quad\|I_{S_1}(a_t^2+b_t^2)\circ n^{-1}X^T(X r_t-w)\|^2 \\
        	&\leq \|(I_{S_1} |\beta_t| \circ n^{-1}X^T(X r_t-w)\|^2 +s \alpha^4 \leq  c \|r_{t,S_1}\|^2+c'  (\alpha^2+\sigma^2 \frac{Ms_1}{n}).
        	\end{align*}
        	
        	For the order three terms, using inequality $\|I_{S_1} (a_t^2+b_t^2)\circ (n^{-1}X^TX r_t)\|_{\infty}\lesssim 1$, we obtain:
        	\begin{align*}
        	&\quad\langle \beta_t\circ (n^{-1}X^TX r_t)^2 ,I_{S_1} (a_t^2+b_t^2)\circ (n^{-1}X^TX r_t)\rangle\\
        	&\leq c \|(I_{S_1}\sqrt{\beta_t}) \circ n^{-1}X^T(X r_t-w)\|^2 \leq  c \|r_{t,S_1}\|^2+c'  (\alpha^2+\sigma^2 \frac{Ms_1}{n}).
        	\end{align*} 
        	
        	For the order four terms, we have 
        	\begin{align*}
        	&\|I_{S_1} \beta_t\circ (n^{-1}X^TX r_t)^2\|^2\leq c \|I_{S_1} \sqrt{\beta_t} (n^{-1}X^TX r_t)\|^2  \leq  c \|r_{t,S_1}\|^2+c'  (\alpha^2+\sigma^2 \frac{Ms_1}{n}).
        	\end{align*}
        	
        	Putting all pieces together, and using our condition $\eta \lesssim 1$, we can obtain
        	\begin{align*}
        	\|I_{S_1} r_{t+1}\|^2 &\leq \|I_{S_1} r_{t}\|^2-2\eta \,\Big\|\,r_{t,S_1}\circ \sqrt{|\beta_{t,S_1}|}\,\Big\|^2 \\
        	&+c  (\alpha^2+\sigma^2 \frac{Ms_1}{n})+4 \eta \delta \|r_{t,S_1}\|^2.
        	\end{align*} 
        	Since for $(a_t,b_t)\in\Theta_{LR}^G$, we have $\|r_{t,S_1}\circ \sqrt{|\beta_{t,S_1}|}\|^2\geq \|r_{t,S_1}\|^2\cdot m/2 $. Therefore, under our conditions on $(\delta,\eta)$ with $\delta  \kappa \leq 1/4$, we have:
        	\begin{equation}
        	\|\beta_{t+1,S_1}-\beta^\ast_{S_1}\|^2\leq(1-\eta m)\|\beta_{t,S_1}-\beta^\ast_{S_1}\|^2+c  (\alpha^2+\sigma^2 \frac{Ms_1}{n}) .
        	\end{equation} 
        \end{proof}
        \subsection{Proof of Theorem~\ref{thm2}} 
        As mentioned after the theorem in the paper, we divide the convergence into two stages. In the first ``burn-in'' stage, we show that the smallest component of the strong signal part $z_t$ increases at an exponential rate in $t$ until hitting $m/2$, 
        \begin{align*}
        \mbox{sign}(\beta^\ast_{S_1})\circ \beta_{t,S_1} \geq \min\Big\{\frac{m}{2},\, \Big(1+\eta m/4\Big)^t\,\alpha^2 -  \,\alpha^2\Big\}\, \mathbf 1\in\mathbb R^{s_1}, \ \text{when}\  \theta_{s_1}(\beta_{t,S_1})\leq {m}/2.
        \end{align*}
        In the second stage, after iterate $\beta_t$ enters $\Theta_{LR}$, we have the geometric convergence up to some high-order error term
        \begin{align*}
        \|\beta_{t+1,S_1}-\beta^\ast_{S_1}\|^2\leq(1-\eta m)\|\beta_{t,S_1}-\beta^\ast_{S_1}\|^2+c  (\alpha^2+\sigma^2 \frac{Ms_1}{n}),  \ \text{when}\  \theta_{s_1}(\beta_{t,S_1})\geq {m}/2,
        \end{align*}
        When $t\leq \Theta(\log (\sqrt{m}/\alpha)/(\eta m))$, the convergence is in the first stage; and proposition~\ref{pro3.1} implies that the iterate $\beta_t$ enters $\Theta_{LR}$ (corresponding to the second stage) in at most $\Theta(\log (1/\alpha)/(\eta m))$ iterations.  proposition~\ref{pro3.1} and proposition~\ref{pro3.4} together imply for any $t=\Theta(\log ({1/\alpha})/(\eta m))$, 
        \begin{equation}
        \|\beta_{t,S_1}-\beta^\ast_{S_1}\|^2 \leq c (\alpha^2+\sigma^2 \frac{Ms_1}{n}).
        \end{equation}
        For any $T_2>t>T_1$, since $\|\beta_{t,S_1^c}\| \lesssim 1/p$ is still controlled, combined with bound~\eqref{decomp}, we have:
        \begin{align*}
        \|\beta-\beta^\ast\|^2 \leq c (\alpha^2+\sigma^2 \frac{Ms_1}{n} + \sigma^2 \frac{s_2 \log p}{n}).
        \end{align*}

        \section{Simulations Beyond RIP Conditions}
        
        In this section, we conduct some numerical experiments to illustrate the performance of our method. In our first example, we study the performance as we change the initialization under settings where the RIP may or may not hold. In the second example, we study the behavior of the algorithm when the null space property \citep{cohen2009compressed} is violated, under which the $\ell_1$-norm minimizer in the basis pursuit problem differs from the sparsest solution (the $\ell_0$-norm minimizer).

        \begin{figure}[t]\label{fig:20}
        	\subfigure[Estimation Error vs initialization]{ \includegraphics[width=0.5\textwidth]{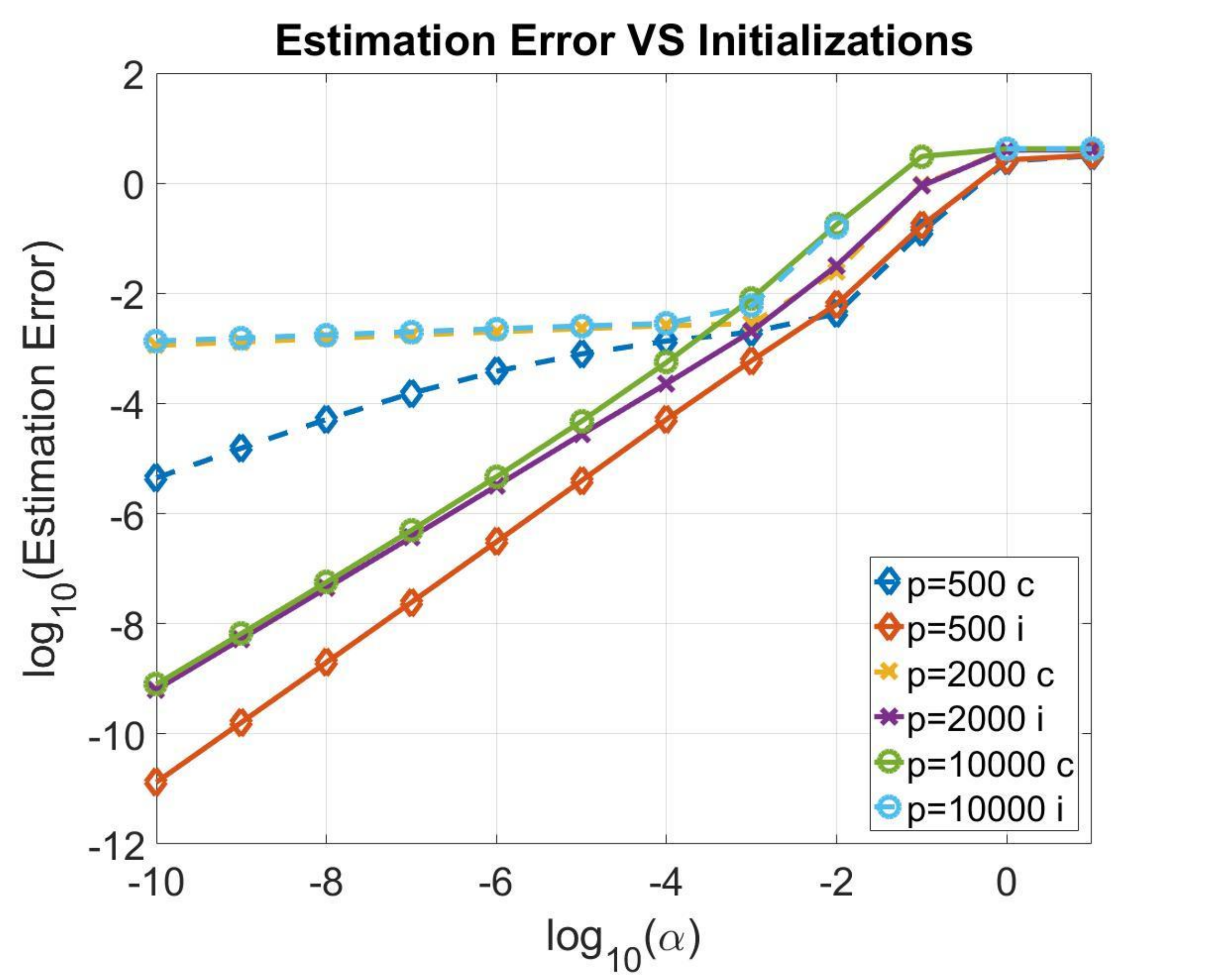}}
        	\subfigure[Signal components trajectories]      {\includegraphics[width=0.5\textwidth]{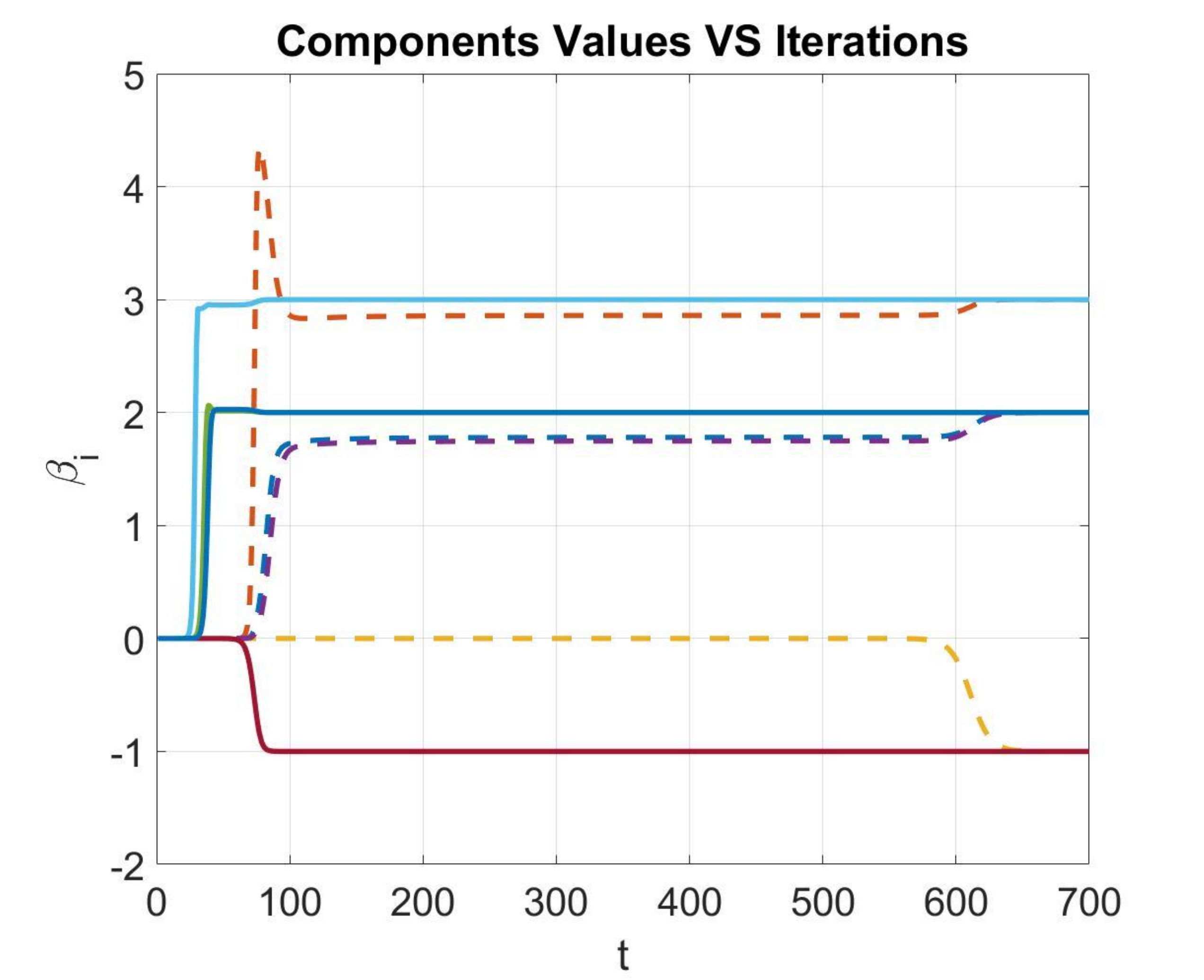}}
        	\caption{Panel (a) is a log-log plot of estimation error $\|\widehat \beta-\beta^\ast\|_2$ versus initialization level $\alpha$. Panel (b) shows trajectories of signal components ($\beta_1$ to $\beta_4$) at $(n,p)=(200,2000)$. In both plots, solid curves correspond to the independent (i) design, and dashed curves the correlated (c) design. }\label{fig:200}
        \end{figure}
        
        We consider two classes of random designs $X\in\mathbb R^{n\times p}$ with rows $X^{(i)}\overset{iid}{\sim}\mathcal N(0,\,\Sigma)\in \mathbb R^p$, where population p.s.d~covariance matrix $\Sigma\in\mathbb R^{p\times p}$ is
        \begin{eqnarray*}
        	\mbox{independent design:}\quad \Sigma_{jk}=\mathbb I(j=k), \quad\mbox{or}\quad \mbox{correlated design:}\quad \Sigma_{jk}=0.5+0.5\,\mathbb I(j=k),
        \end{eqnarray*}
        for $i,j,k=1,\ldots,p$, where $\mathbb I(A)$ denotes the indicator function of event $A$.
        We choose sample size $n=200$, sparsity level $s=4$, and signal dimension $p\in\{500, 2000,10000\}$. For the independent (correlated) design, the RIP is satisfied (fails) with high probability for some small $\delta>0$. In both scenarios, we choose true signal $\beta^\ast=(-1,2,2,3)^T\in\mathbb R^p$, and set $y=X\beta^\ast$. When implementing gradient descent, we choose step size $\eta=0.2\ (0.1)$ for the independent (correlated) design, $\alpha\in\{10^{-10},10^{-9},\ldots,10^{1}\}$, and stopping threshold $\epsilon=0.01\alpha$. Figure~\ref{fig:200} shows the estimation error $\|\widehat \beta-\beta^\ast\|_2$ versus $\alpha$ in log-log plots. As we can see, they tend to have a linear trend under the log-scale, which is consistent with our theoretical error bound estimate in Section 4.
        In addition, in the correlated design scenario where the RIP does not hold, the algorithm can still recover $\beta^\ast$ as $\alpha\to0$, albeit under a slower convergence (due to a smaller allowable step size and a larger condition number of $X$). This observation provides evidence of the correctness of our informal statement made at the beginning of Section 3, even without RIP condition. We leave the proof of this conjecture open.
        
        In this example, we study the empirical performance of the algorithm when the least $\ell_1$-norm in the basis pursuit problem~\eqref{Eqn:CS} is not the sparsest solution of $X\beta=y$ (the null space property is violated). In particular, we choose
        \begin{eqnarray*}
        	X=\begin{bmatrix} 0.2 & 1 & 0 \\ 0.2 & 0 &-1 \end{bmatrix},\quad \beta^\ast= \begin{bmatrix} 5\\ 0 \\0 \end{bmatrix}, \quad\mbox{and}\quad y=\begin{bmatrix} 1\\ 1 \end{bmatrix},
        \end{eqnarray*}
        so that $X\beta^\ast =y$. It is easy to verify that for this example, the sparsest solution of $X\beta=y$ is $\beta^\ast$, while the least $\ell_1$-norm solution is $\beta^\dagger = [0,1,-1]^T$. We use the same setting as before for implementing the algorithm with $\alpha\in\{10^{-10},10^{-5},10^{-3},10^{-1},10^0,10^1\}$. Table~\ref{tb1} reports final outputs $\beta=(\beta_1,\beta_2,\beta_3)^T$ of the algorithm. Due to our intuition in Section 3, as expected, the algorithm still converges to the least $\ell_1$-norm solution $\beta^\dagger$ instead of the least $\ell_0$-norm solution $\beta^\ast$. Again, the estimation error decreases as the initialization level $\alpha$ decreases. 
        \begin{table}\caption{\label{tb1} Convergent point without null space property} \begin{tabular*}{\textwidth}{|c@{\extracolsep{\fill}}cccccc|} 
        		\hline
        		$\alpha$ & 1e-10 & 1e-5 & 1e-3 & 0.1 & 1 & 10 \\ 
        		\hline
        		$\beta_1$ & 7.433e-13 & 5.703e-7 & 1.289e-4 & 2.884e-2 & 2.987e-1 & 8.823e-1 \\ 
        		\hline
        		$1-\beta_2$ & 1.492e-13 & 1.141e-7 & 2.577e-5 & 5.769e-3 & 5.974e-2 & 1.765e-1\\
        		\hline
        		$1+\beta_3$ & 1.492e-13 & 1.141e-7 & 2.577e-5 & 5.769e-3 & 5.974e-2 & 1.765e-1 \\
        		\hline 
        	\end{tabular*}
        \end{table}
        
        \section{Simulation results for Algorithm~\ref{alg:1}}
        We adopt the simulation settings in S1-S4 for strong signals, except that now $\sigma=0.5$. For Algorithm~\ref{alg:1}, we use MCP as the pilot estimator and then use gradient descent for post-debiasing. The below is the median of standardized RMSE of the repeated $50$ simulations. The results show that Algorithm~\ref{alg:1}  could achieve comparable results with Algorithm~\ref{alg2}. 
        \begin{table}[ht]
        	\caption{Performance comparison between Algorithm 2 and post-debiasing Algorithm 3.}\label{tab:1}
        	\centering
        	\begin{tabular}{r|lllllllllllll}
        		\hline
        		Case  &S1 & S2 & S3 & S4  \\
        		\hline
        		
        		{ Algorithm~\ref{alg2}}           & 0.00530 & 0.00590 & 0.0126 & 0.0209 \\

        		{Algorithm~\ref{alg:1}}             & 0.00839 & 0.00827  & 0.0159 & 0.0173  \\
        		
        		\hline
        	\end{tabular}
        \end{table}

        \section{Point-to point Comparisons with \cite{vaskevicius2019implicit}}
        
        \begin{enumerate}
        	\item \textbf{Theoretical results.} We both consider investigating implicit regularizations in high dimensional regressions under the RIP assumption. We adopt parametrization $\beta=g \circ l$ as the problem formulation since it has nice interpretations both from gradient dynamic systems and implicit regularizations in $\ell_2$ settings. In the proof for general signals, we introduced the equivalent parametrization $\beta=a^2-b^2$, where $a=(g+l)/2$ and $b=(g-l)/2$, as a proof technique to facilitate our proof (which is NOT inspired by their work, but appeared in our first version before their work appeared). However, their work directly applies the parametrization $\beta=a^2-b^2$ without providing detailed intuitions and connections.
        	

        	Compared with previous work \citep{li2018algorithmic} in the area, both our works make three new contributions in the theoretical analysis: 1. adopting the parametrization $\beta=a^2-b^2$ such that linear regression models can be properly analyzed; 2. dividing the signals into strong and weak ones and connecting signal strength with the minimal number of iterations in order to recover them; 3. showing that the final error rate can be adaptively improved via early stopping.  For the detailed proofs, we follow \citep{li2018algorithmic} by also analyzing the convergence under the RIP assumption, while \cite{vaskevicius2019implicit} developed an improved analysis for handling noises but still requires the RIP condition. Below we compare the assumptions made in the two works.
        	
        	Our conditions on the RIP constant, step size and signal strengths are more stringent than theirs. There are two reasons to explain this difference. First, our target is to show that the estimator achieves the benchmark optimal rate of $\sqrt{s_1/n}+\sqrt{s_2 \log p/n}$ ($s_1,s_2$ are number of strong and weak signals) -- where the rate becomes parameter root $n$ rate if there is no weak signals (i.e.~$s_2=0$), which is strictly better than the LASSO. To achieve this goal, we need to make some stronger assumptions. In comparison, their rate is $\sqrt{s_1 \log(s_1)/n}+\sqrt{s_2 \log p/n}$, which is slightly worse in terms of the dependence on the sparsity level $s_1$.
        	Second,  \cite{vaskevicius2019implicit} developed an improved analysis which makes the conditions less stringent than ours.  However, our more stringent conditions do not affect the application of our method if we mainly apply implicit regularization as a way for refining a rough initial estimator through our newly added two-stage debiasing approach.

        	Another important theoretical distinction is the different stabilized optimal stopping regions $[a,b]$: $t \in [a,b]$ when the estimator achieves the optimal rate. Specifically, the theorem  in \cite{vaskevicius2019implicit} implies that both $a$ and $b$ are at the same rate $a,b = O(\log(1/\alpha) \kappa)$ (c.f. Corollary 2,3 in \cite{vaskevicius2019implicit},  where $\alpha$ is the initial value and $\kappa$ is the signal conditional number). The results also mean that $b$ is independent of $\sigma$ when $\kappa$ is at the constant level. In contrast, our theorem indicates that the optimal region is $ a=O(\log(1/\alpha) \kappa/m)$ and $ b=O( \log(1/\alpha) \kappa/(\sigma \sqrt{\log p/n}))$ ($m$ is the smallest signal), where the smaller value in $\sigma$ will result in a larger $b$, therefore a more expanded optimal stopping region.  Our theoretical results match the simulation studies in section A.6 in the supplementary material better. Moreover, similar results (where $a$ and $b$ have different rates) for the optimal stopping region can also be seen in the related literature (e.g., Theorem 3.6 in \cite{fan2020understanding}).
        	
        	\item \textbf{Computation.}
        	
        	For our vanilla algorithm, we did not elaborate much on the computational aspects, while their work designed an adaptive algorithm that performs better in computation, which is $O(\sqrt{n})$ in the number of iteration for convergence. In the revised paper, we added a different algorithm to achieve the same computational efficiency based on pre-processing the data. 
        	Moreover, we want to point out that the improved algorithm in~\cite{vaskevicius2019implicit} requires applying different step sizes (that also depend on time) for updating different components in $\beta$. Their variation on the algorithm loses the implicit regularization interpretation: its accompanied (continuous-time) limiting gradient dynamical system will effectively multiply an extra diagonal matrix in front for both $\dot{g}(t)$ and $\dot{l}(t)$ as compared to our limiting dynamical system (4) in the paper. Consequently, it is not clear whether their improved algorithm still converges to the minimal $\ell_1$ solution, for example, in the noiseless setting without the RIP condition. This means that their improved algorithm may truly rely on the RIP condition, not like our method where the RIP condition is mainly made for technical purposes (c.f.~the heuristic illustration in Section~2.4 and numerical results in Section D of the supplement for details).
        	
        	\item  \textbf{Methodology.} From the gradient dynamic systems point of view, our work has many extensions in the proposed methodology. First,  assigning different weights (step sizes) to different coordinates will result in similar results as the adaptive lasso; moreover, with the help of one reviewer, we also bring the extension to elastic-net, where an additional $\ell_2$ penalty  $\lambda \|g \circ l\|^2$ is incorporated into the objective function. Second, we also provide an empirical comparison of different stopping criteria. Third, we discuss how to conduct variable selection based on the output (via proper thresholding (Appendix A.2). Last but not least, we propose the use of our implicit regularization method as a refining technique for improving explicit regularization based estimators, which boosts their convergence rate from $\sqrt{s\log p/n}$ to $\sqrt{s/n}$ when there is no weak signal.
        	
        	\item \textbf{Data analysis.} Both our works show numerical evidence to support our main theorems. In addition to that, both of us provide simulations that implicit regularization still works when a much relaxed restricted eigenvalue condition on the design matrix instead of the RIP holds. In our work, we also show that when null space property does not hold, our algorithm can not converge to the sparest solution even under the noiseless case, which provides high-level intuitions that the sparsity of the solution is induced from the $\ell_1$ norm instead of $\ell_0$ type of method (like forward-backward regression). The high-level messages are closely connected with our dynamic system, and $\ell_2$ norm regularized interpretations.

        	\item \textbf{Others.}
        	In our work, we provide an analysis of the optimization landscape in the noiseless regime (Lemma 1 \& 2) by showing that the non-convex objective function after the Hadamard product parametrization does not have local maximums, and all its local minimums are global. Consequently, the gradient method for optimizing the objective function with random initialization almost surely converges to a global minimum. We also perform simulations and real data analysis to study variable selection and illustrate the success of some empirical rules (using cross-validated $\lambda$ from the lasso) for achieving good variable selection performance. In comparison, \cite{vaskevicius2019implicit} didn't discuss/study variable selection.
        \end{enumerate}

        \section{Discussion}
 We discuss several important open problems on implicit regularization as our future directions. First, our theory heavily relies on the RIP of the design matrix, which is relatively strong comparing to the restricted eigenvalue condition \citep{bickel2009simultaneous} as the minimal possible assumption in the literature. It would be interesting to theoretically investigate whether our results remain valid without the RIP. Second, it is of practical importance to study whether any computationally-efficient early stopping rule based on certain data-driven model complexity measures rather than the cross validation method can be applied to reliably and robustly select the tuning parameters, such as the iteration number and the step size in our algorithm.    
        \section*{Acknowledgment}       
        Dr. Zhao's work was partially supported by National Science Foundation grant CCF-1934904. Dr. Yang's work was partially supported by National Science Foundation grant DMS-1810831. We would like to thank the editor, associate editor and two anonymous reviewers from Biometrika for their careful comments and helpful suggestions that significantly improved the quality of the paper.        
        \section*{Supplementary material}
        Supplementary material contains detailed discussions on dynamic system interpretation of the implicit regularization, extension to adaptive step size and elastic net, proof of the theorems, extra simulations results beyond Restricted Isometry Property, some reproducible simulation codes and a detailed comparison with \cite{vaskevicius2019implicit}.
        \bibliographystyle{apalike}        
        \bibliography{implicit_abbre,implicit}        

\begin{thebibliography}{}

\bibitem[Bickel et~al., 2009]{bickel2009simultaneous}
Bickel, P.~J., Ritov, Y., Tsybakov, A.~B., et~al. (2009).
\newblock Simultaneous analysis of lasso and dantzig selector.
\newblock {\em Ann. Statist.}, 37(4):1705--1732.

\bibitem[Breheny and Huang, 2011]{breheny2011coordinate}
Breheny, P. and Huang, J. (2011).
\newblock Coordinate descent algorithms for nonconvex penalized regression,
  with applications to biological feature selection.
\newblock {\em Ann. Appl. Statist.}, 5(1):232.

\bibitem[B{\"u}hlmann et~al., 2014]{buhlmann2014high}
B{\"u}hlmann, P., Kalisch, M., and Meier, L. (2014).
\newblock High-dimensional statistics with a view toward applications in
  biology.
\newblock {\em Annu. Rev. Stat. Appl.}, 1:255--278.

\bibitem[Candes and Tao, 2007]{candes2007dantzig}
Candes, E. and Tao, T. (2007).
\newblock The dantzig selector: Statistical estimation when p is much larger
  than n.
\newblock {\em Ann. Statist.}, 35(6):2313--2351.

\bibitem[Candes, 2008]{candes2008restricted}
Candes, E.~J. (2008).
\newblock The restricted isometry property and its implications for compressed
  sensing.
\newblock {\em C. R. Math.}, 346(9-10):589--592.

\bibitem[Chen et~al., 2001]{chen2001atomic}
Chen, S.~S., Donoho, D.~L., and Saunders, M.~A. (2001).
\newblock Atomic decomposition by basis pursuit.
\newblock {\em SIAM Rev.}, 43(1):129--159.

\bibitem[Cohen et~al., 2009]{cohen2009compressed}
Cohen, A., Dahmen, W., and DeVore, R. (2009).
\newblock Compressed sensing and best k-term approximation.
\newblock {\em J. Am. Math. Soc.}, 22:211--231.

\bibitem[Donoho, 2006]{donoho2006compressed}
Donoho, D.~L. (2006).
\newblock Compressed sensing.
\newblock {\em IEEE T. Inform. Theory}, 52(4):1289--1306.

\bibitem[Efron, 2004]{efron2004estimation}
Efron, B. (2004).
\newblock The estimation of prediction error: covariance penalties and
  cross-validation.
\newblock {\em J. Am. Statist. Ass.}, 99(467):619--632.

\bibitem[Fan and Li, 2001]{fan2001variable}
Fan, J. and Li, R. (2001).
\newblock Variable selection via nonconcave penalized likelihood and its oracle
  properties.
\newblock {\em J. Am. Statist. Ass.}, 96(456):1348--1360.

\bibitem[Fan and Lv, 2008]{fan2008sure}
Fan, J. and Lv, J. (2008).
\newblock Sure independence screening for ultrahigh dimensional feature space.
\newblock {\em J. R. Statist. Soc. B}, 70(5):849--911.

\bibitem[Fan et~al., 2020]{fan2020understanding}
Fan, J., Yang, Z., and Yu, M. (2020).
\newblock Understanding implicit regularization in over-parameterized nonlinear
  statistical model.
\newblock {\em arXiv preprint arXiv:2007.08322}.

\bibitem[Friedman et~al., 2010]{friedman2010regularization}
Friedman, J., Hastie, T., and Tibshirani, R. (2010).
\newblock Regularization paths for generalized linear models via coordinate
  descent.
\newblock {\em J. Stat. Softw.}, 33(1):1.

\bibitem[Gunasekar et~al., 2018]{gunasekar2018characterizing}
Gunasekar, S., Lee, J., Soudry, D., and Srebro, N. (2018).
\newblock Characterizing implicit bias in terms of optimization geometry.
\newblock In {\em Int. Conf. Mach. Learn.}, pages 1827--1836.

\bibitem[Gunasekar et~al., 2017]{gunasekar2017implicit}
Gunasekar, S., Woodworth, B.~E., Bhojanapalli, S., Neyshabur, B., and Srebro,
  N. (2017).
\newblock Implicit regularization in matrix factorization.
\newblock In {\em Adv. Neur. In.}, pages 6152--6160.

\bibitem[Hoff, 2017]{hoff2017lasso}
Hoff, P.~D. (2017).
\newblock Lasso, fractional norm and structured sparse estimation using a
  hadamard product parametrization.
\newblock {\em Comput. Stat. Data An.}, 115:186--198.

\bibitem[Jin and Ke, 2016]{jin2016rare}
Jin, J. and Ke, Z.~T. (2016).
\newblock Rare and weak effects in large-scale inference: methods and phase
  diagrams.
\newblock {\em Stat. Sinica}, pages 1--34.

\bibitem[Lee et~al., 2016]{lee2016gradient}
Lee, J.~D., Simchowitz, M., Jordan, M.~I., and Recht, B. (2016).
\newblock Gradient descent only converges to minimizers.
\newblock In {\em Conference on Learning Theory}, pages 1246--1257.

\bibitem[Li et~al., 2018]{li2018algorithmic}
Li, Y., Ma, T., and Zhang, H. (2018).
\newblock Algorithmic regularization in over-parameterized matrix sensing and
  neural networks with quadratic activations.
\newblock In {\em Conference On Learning Theory}, pages 2--47.

\bibitem[Raskutti et~al., 2014]{raskutti2014early}
Raskutti, G., Wainwright, M.~J., and Yu, B. (2014).
\newblock Early stopping and non-parametric regression: an optimal
  data-dependent stopping rule.
\newblock {\em J. Mach. Learn. Res.}, 15(1):335--366.

\bibitem[Rudelson et~al., 2013]{rudelson2013hanson}
Rudelson, M., Vershynin, R., et~al. (2013).
\newblock Hanson-wright inequality and sub-gaussian concentration.
\newblock {\em Electron. Commun. Prob.}, 18.

\bibitem[Soudry et~al., 2018]{soudry2018implicit}
Soudry, D., Hoffer, E., Nacson, M.~S., Gunasekar, S., and Srebro, N. (2018).
\newblock The implicit bias of gradient descent on separable data.
\newblock {\em J. Mach. Learn. Res.}, 19(1):2822--2878.

\bibitem[Stein, 1981]{stein1981estimation}
Stein, C.~M. (1981).
\newblock Estimation of the mean of a multivariate normal distribution.
\newblock {\em Ann. Statist.}, pages 1135--1151.

\bibitem[Tibshirani, 1996]{tibshirani1996regression}
Tibshirani, R. (1996).
\newblock Regression shrinkage and selection via the lasso.
\newblock {\em J. R. Statist. Soc. B}, pages 267--288.

\bibitem[Vaskevicius et~al., 2019]{vaskevicius2019implicit}
Vaskevicius, T., Kanade, V., and Rebeschini, P. (2019).
\newblock Implicit regularization for optimal sparse recovery.
\newblock In {\em Adv. Neur. In.}, pages 2968--2979.

\bibitem[Vito et~al., 2005]{vito2005learning}
Vito, E.~D., Rosasco, L., Caponnetto, A., Giovannini, U.~D., and Odone, F.
  (2005).
\newblock Learning from examples as an inverse problem.
\newblock {\em J. Mach. Learn. Res.}, 6(May):883--904.

\bibitem[Wainwright, 2009]{wainwright2009information}
Wainwright, M.~J. (2009).
\newblock Information-theoretic limits on sparsity recovery in the
  high-dimensional and noisy setting.
\newblock {\em IEEE T. Inform. Theory}, 55(12):5728--5741.

\bibitem[Yao et~al., 2007]{yao2007early}
Yao, Y., Rosasco, L., and Caponnetto, A. (2007).
\newblock On early stopping in gradient descent learning.
\newblock {\em Constr. Approx.}, 26(2):289--315.

\bibitem[Zhang, 2010]{zhang2010nearly}
Zhang, C.-H. (2010).
\newblock Nearly unbiased variable selection under minimax concave penalty.
\newblock {\em Ann. Statist.}, 38(2):894--942.

\bibitem[Zhang and Yu, 2005]{zhang2005boosting}
Zhang, T. and Yu, B. (2005).
\newblock Boosting with early stopping: Convergence and consistency.
\newblock {\em Ann. Statist.}, 33(4):1538--1579.

\bibitem[Zhao et~al., 2018]{zhao2018pathwise}
Zhao, T., Liu, H., and Zhang, T. (2018).
\newblock Pathwise coordinate optimization for sparse learning: Algorithm and
  theory.
\newblock {\em Ann. Statist.}, 46(1):180--218.

\bibitem[Zou, 2006]{zou2006adaptive}
Zou, H. (2006).
\newblock The adaptive lasso and its oracle properties.
\newblock {\em J. Am. Statist. Ass.}, 101(476):1418--1429.

\bibitem[Zou et~al., 2007]{zou2007degrees}
Zou, H., Hastie, T., and Tibshirani, R. (2007).
\newblock On the “degrees of freedom” of the lasso.
\newblock {\em Ann. Statist.}, 35(5):2173--2192.

\end{thebibliography}


\begin{thebibliography}{}

\bibitem[Beck and Teboulle, 2009]{beck2009fast}
Beck, A. and Teboulle, M. (2009).
\newblock A fast iterative shrinkage-thresholding algorithm for linear inverse
  problems.
\newblock {\em SIAM Journal on Imaging Sciences}, 2(1):183--202.

\bibitem[Breheny and Huang, 2011]{breheny2011coordinate}
Breheny, P. and Huang, J. (2011).
\newblock Coordinate descent algorithms for nonconvex penalized regression,
  with applications to biological feature selection.
\newblock {\em The annals of applied statistics}, 5(1):232.

\bibitem[B{\"u}hlmann et~al., 2014]{buhlmann2014high}
B{\"u}hlmann, P., Kalisch, M., and Meier, L. (2014).
\newblock High-dimensional statistics with a view toward applications in
  biology.
\newblock {\em Annual Review of Statistics and Its Application}, 1:255--278.

\bibitem[Candes and Tao, 2007]{candes2007dantzig}
Candes, E. and Tao, T. (2007).
\newblock The dantzig selector: Statistical estimation when p is much larger
  than n.
\newblock {\em The Annals of Statistics}, 35(6):2313--2351.

\bibitem[Candes, 2008]{candes2008restricted}
Candes, E.~J. (2008).
\newblock The restricted isometry property and its implications for compressed
  sensing.
\newblock {\em Comptes Rendus Mathematique}, 346(9-10):589--592.

\bibitem[Chen et~al., 2001]{chen2001atomic}
Chen, S.~S., Donoho, D.~L., and Saunders, M.~A. (2001).
\newblock Atomic decomposition by basis pursuit.
\newblock {\em SIAM Review}, 43(1):129--159.

\bibitem[Cohen et~al., 2009]{cohen2009compressed}
Cohen, A., Dahmen, W., and DeVore, R. (2009).
\newblock Compressed sensing and best k-term approximation.
\newblock {\em Journal of the {A}merican {M}athematical {S}ociety},
  22:211--231.

\bibitem[Donoho, 2006]{donoho2006compressed}
Donoho, D.~L. (2006).
\newblock Compressed sensing.
\newblock {\em IEEE Transactions on Information Theory}, 52(4):1289--1306.

\bibitem[Efron, 2004]{efron2004estimation}
Efron, B. (2004).
\newblock The estimation of prediction error: covariance penalties and
  cross-validation.
\newblock {\em Journal of the American Statistical Association},
  99(467):619--632.

\bibitem[Fan and Li, 2001]{fan2001variable}
Fan, J. and Li, R. (2001).
\newblock Variable selection via nonconcave penalized likelihood and its oracle
  properties.
\newblock {\em Journal of the American statistical Association},
  96(456):1348--1360.

\bibitem[Fan and Lv, 2008]{fan2008sure}
Fan, J. and Lv, J. (2008).
\newblock Sure independence screening for ultrahigh dimensional feature space.
\newblock {\em Journal of the Royal Statistical Society: Series B (Statistical
  Methodology)}, 70(5):849--911.

\bibitem[Friedman et~al., 2010]{friedman2010regularization}
Friedman, J., Hastie, T., and Tibshirani, R. (2010).
\newblock Regularization paths for generalized linear models via coordinate
  descent.
\newblock {\em Journal of statistical software}, 33(1):1.

\bibitem[Gunasekar et~al., 2018]{pmlr-v80-gunasekar18a}
Gunasekar, S., Lee, J., Soudry, D., and Srebro, N. (2018).
\newblock Characterizing implicit bias in terms of optimization geometry.
\newblock In Dy, J. and Krause, A., editors, {\em Proceedings of the 35th
  International Conference on Machine Learning}, volume~80 of {\em Proceedings
  of Machine Learning Research}, pages 1832--1841, Stockholmsm{\"a}ssan,
  Stockholm Sweden. PMLR.

\bibitem[Gunasekar et~al., 2017]{gunasekar2017implicit}
Gunasekar, S., Woodworth, B.~E., Bhojanapalli, S., Neyshabur, B., and Srebro,
  N. (2017).
\newblock Implicit regularization in matrix factorization.
\newblock In {\em Advances in Neural Information Processing Systems}, pages
  6152--6160.

\bibitem[Hoff, 2017]{hoff2017lasso}
Hoff, P.~D. (2017).
\newblock Lasso, fractional norm and structured sparse estimation using a
  hadamard product parametrization.
\newblock {\em Computational Statistics \& Data Analysis}, 115:186--198.

\bibitem[Jin and Ke, 2016]{jin2016rare}
Jin, J. and Ke, Z.~T. (2016).
\newblock Rare and weak effects in large-scale inference: methods and phase
  diagrams.
\newblock {\em Statistica Sinica}, pages 1--34.

\bibitem[Lee et~al., 2016]{lee2016gradient}
Lee, J.~D., Simchowitz, M., Jordan, M.~I., and Recht, B. (2016).
\newblock Gradient descent only converges to minimizers.
\newblock In {\em Conference on Learning Theory}, pages 1246--1257.

\bibitem[Li et~al., 2018]{li2018algorithmic}
Li, Y., Ma, T., and Zhang, H. (2018).
\newblock Algorithmic regularization in over-parameterized matrix sensing and
  neural networks with quadratic activations.
\newblock In {\em Conference On Learning Theory}, pages 2--47.

\bibitem[Raskutti et~al., 2014]{raskutti2014early}
Raskutti, G., Wainwright, M.~J., and Yu, B. (2014).
\newblock Early stopping and non-parametric regression: an optimal
  data-dependent stopping rule.
\newblock {\em The Journal of Machine Learning Research}, 15(1):335--366.

\bibitem[Soudry et~al., 2018]{soudry2018implicit}
Soudry, D., Hoffer, E., Nacson, M.~S., Gunasekar, S., and Srebro, N. (2018).
\newblock The implicit bias of gradient descent on separable data.
\newblock {\em The Journal of Machine Learning Research}, 19(1):2822--2878.

\bibitem[Stein, 1981]{stein1981estimation}
Stein, C.~M. (1981).
\newblock Estimation of the mean of a multivariate normal distribution.
\newblock {\em The annals of Statistics}, pages 1135--1151.

\bibitem[Tibshirani, 1996]{tibshirani1996regression}
Tibshirani, R. (1996).
\newblock Regression shrinkage and selection via the lasso.
\newblock {\em Journal of the Royal Statistical Society. Series B
  (Methodological)}, pages 267--288.

\bibitem[Vito et~al., 2005]{vito2005learning}
Vito, E.~D., Rosasco, L., Caponnetto, A., Giovannini, U.~D., and Odone, F.
  (2005).
\newblock Learning from examples as an inverse problem.
\newblock {\em Journal of Machine Learning Research}, 6(May):883--904.

\bibitem[Wainwright, 2009]{wainwright2009information}
Wainwright, M.~J. (2009).
\newblock Information-theoretic limits on sparsity recovery in the
  high-dimensional and noisy setting.
\newblock {\em IEEE Transactions on Information Theory}, 55(12):5728--5741.

\bibitem[Yao et~al., 2007]{yao2007early}
Yao, Y., Rosasco, L., and Caponnetto, A. (2007).
\newblock On early stopping in gradient descent learning.
\newblock {\em Constructive Approximation}, 26(2):289--315.

\bibitem[Zhang, 2010]{zhang2010nearly}
Zhang, C.-H. (2010).
\newblock Nearly unbiased variable selection under minimax concave penalty.
\newblock {\em The Annals of statistics}, 38(2):894--942.

\bibitem[Zhang and Yu, 2005]{zhang2005boosting}
Zhang, T. and Yu, B. (2005).
\newblock Boosting with early stopping: Convergence and consistency.
\newblock {\em The Annals of Statistics}, 33(4):1538--1579.

\bibitem[Zhao et~al., 2018]{zhao2018pathwise}
Zhao, T., Liu, H., and Zhang, T. (2018).
\newblock Pathwise coordinate optimization for sparse learning: Algorithm and
  theory.
\newblock {\em The Annals of Statistics}, 46(1):180--218.

\bibitem[Zou, 2006]{zou2006adaptive}
Zou, H. (2006).
\newblock The adaptive lasso and its oracle properties.
\newblock {\em Journal of the American statistical association},
  101(476):1418--1429.

\end{thebibliography}
\end{document}